\documentclass[a4paper]{article}
\usepackage{amsmath, amsfonts, amsthm, amssymb, amscd}
\usepackage{latexsym}
\usepackage{graphicx, mathptmx}
\usepackage{mathtools, stackengine}
\usepackage[mathcal]{euscript} 
\usepackage{newtxtext, newtxmath, dsfont, mathrsfs, accents, verbatim} 
\usepackage{enumerate, color}
\usepackage{epsfig}
\usepackage{delarray}
\usepackage[titletoc,title]{appendix}
\newtheorem{theorem}{Theorem}
\newtheorem{definition}{Definition}
\newtheorem{remark}{Remark}
\newtheorem{lemma}{Lemma}
\newtheorem{proposition}{Proposition}
\newtheorem{corollary}{Corollary}
\newtheorem*{acknowledgements}{Acknowledgement}

\numberwithin{equation}{section}
\newcommand{\divx}{\mathop{\mathrm{div}}}
\newcommand{\sgn}{\mathop{\mathrm{sgn}}}
\newcommand{\argmin}{\mathop{\mathrm{arg~min}}}
\newcommand{\const}{\mathop{\mathrm{const.}}}
\newcommand{\esssup}{\mathop{\mathrm{ess~sup}}}
\title{Local Lipschitz bounds for solutions to certain singular elliptic equations involving one-Laplacian}
\author{Shuntaro Tsubouchi\footnote{Graduate School of Mathematical Sciences, The University of Tokyo, Japan. \textit{Email}: \texttt{tsubos@ms.u-tokyo.ac.jp}}}
\date{}

\begin{document}
\maketitle
\begin{abstract}
In this paper local Lipschitz regularity of weak solutions to certain singular elliptic equations involving one-Laplacian is studied. Equations treated here also contains another well-behaving elliptic operator such as \(p\)-Laplacian with \(1<p<\infty\). The problem is that one-Laplacian is too singular on degenerate points, what is often called facet, which makes it difficult to obtain even Lipschitz regularity of weak solutions. This difficulty is overcome by making suitable approximation schemes, and by avoiding analysis on facet for approximated solutions. The key estimate is a local a priori uniform Lipschitz estimate for classical solutions to regularized equations, which is proved by Moser's iteration. Another local a priori uniform Lipschitz bounds can also be obtained by De Giorgi's truncation. Proofs of local Lipschitz estimates in this paper are rather classical and elementary in the sense that nonlinear potential estimates are not used at all.
\end{abstract}
\bigbreak
\bigbreak

\section{Introduction and main theorem}\label{Sect Intro}
Let \(\Omega\subset{\mathbb R}^{n}\) be a bounded Lipschitz domain in \(n\)-dimensional Euclidian space, and let \(f\) be a real-valued function on \(\Omega\). We fix constants \(1<p<\infty,\,0<\beta<\infty\). The aim of this paper is to obtain local Lipschitz regularities for solutions to 
\begin{equation}\label{crystal model eq}
-\beta\divx\left(\nabla u/\lvert\nabla u\rvert\right)-\divx\left(\lvert\nabla u \rvert^{p-2}\nabla u\right)\ni f\quad \textrm{ in } \Omega,
\end{equation}
or often simply denoted by \(-\beta\Delta_{1}u-\Delta_{p}u\ni f\) in \(\Omega\). More generally, we consider equations
\begin{equation}\label{1+p Laplacian}
-\beta \divx (\nabla u/\lvert \nabla u\rvert)-\divx\nabla_{z} E_{p}(\nabla u)\ni f\quad \text{\rm in } \Omega,
\end{equation}
where \(E_{p}\) is a real-valued function in \({\mathbb R}^{n}\), such as \(\lvert z\rvert^{p}/p\,\left(z\in{\mathbb R}^{n}\right)\).

\subsection{A typical example and our result}\label{Subsect typical model}
Consider (\ref{crystal model eq}) with \(f=0\). This equation derives from a minimizing problem of the energy functional \[G(u)\coloneqq \beta\int_{\Omega}\lvert\nabla u\rvert\,dx+\frac{1}{p}\int_{\Omega}\lvert\nabla u\rvert^{p}\,dx.\] \(G\) appears as a crystal surface energy, especially for the case \(p=3\). The nonhomogeneous term \(f\) can be regarded here as chemical potential for the crystal surface energy \(G\), in the sense that
\[f=\frac{\delta G}{\delta u}=-\beta\divx\left(\nabla u/\lvert\nabla u\rvert\right)-\divx\left(\lvert\nabla u \rvert^{p-2}\nabla u\right).\]
For details of justifications, see \cite{MR2772127}, \cite{spohn1993surface} and the references given there.
Also in general, equation (\ref{crystal model eq}) comes from a minimizing problem of the energy functional
\[F(u)\coloneqq G(u)-\int_{\Omega}fu\,dx.\]

It is well-known that the diffusion singularity of the operator one-Laplacian, denoted by \(\Delta_{1}\), appears strongly on degenerate points \(\{\nabla u=0\}\), or often called facet. This singularity makes it difficult to consider a term \(\nabla u/\lvert\nabla u\rvert\) in classical sense over facet. Therefore in the first place, when we consider weak solutions (that is, solutions in distributional sense) to (\ref{crystal model eq}),  we face to give a definition of the term \(\nabla u/\lvert\nabla u\rvert\), which should be mathematically valid. The definition of weak solutions is given later in Section \ref{Sect weak solution}.

Also, when it comes to smoothness of solutions, the problem is that elliptic regularity properties of \(\Delta_{1}\) are not understood so much. It is remarkable that diffusion effect of \(\Delta_{1}u=\divx (\nabla u/\lvert\nabla u\rvert)\), unlike that of the singular elliptic operator \(\Delta_{p}u=\divx\left(\lvert\nabla u\rvert^{p-2}\nabla u\right)\,(1<p<2)\), degenerates in the direction normal to levelset surface \cite{MR2376662}.
The diffusion singularity of \(\Delta_{1}\) on facet is connected with the fact that, unlike \(\lvert z\rvert^{p}/p\,\left(z\in{\mathbb R}^{n}\right)\) for some fixed \(1<p<\infty\), the functional \(\lvert z\rvert\,\left(z\in{\mathbb R}^{n}\right)\) loses differentiability at \(0\in{\mathbb R}^{n}\) \cite{MR1712447}. These facts give us the difference between one-Laplacian \(\Delta_{1}\) and \(p\)-Laplacian \(\Delta_{p}\,(1<p<\infty)\) on elliptic regularity. Theorem \ref{Lipschitz bound for Crystal Eq} below reveals that, at least for local Lipschitz regularity, \(p\)-Laplacian plays a dominant role.

\begin{theorem}\label{Lipschitz bound for Crystal Eq}
Let \(u\) be a solution to (\ref{crystal model eq}) in weak sense. Then we have 
\[\lVert\nabla u\rVert_{L^{\infty}(B_{\theta R})}\le C(n,\,p,\,q,\,\beta,\,\theta)\left(1+\lVert f\rVert_{L^{q}(B_{R})}^{1/(p-1)}+R^{-n/p}\lVert\nabla u\rVert_{L^{p}(B_{R})}\right)\]
for any fixed closed ball \(B_{R}\subset \Omega\) with its radius \(0<R\le 1\), any \(2\le n<q\le\infty\) and \(0<\theta<1\).
\end{theorem}
This type of gradient bound estimate has already been given in \cite{krugel2013variational} and \cite{xu2019mathematical}, where nonhomogeneous terms are controlled by \(L^{\infty}\)-data.
The novelty of Theorem \ref{Lipschitz bound for Crystal Eq} is that the nonhomogeneous term \(f\) is controlled by an \(L^{q}\)-datum with \(n<q\le\infty\). We also note that local Lipschitz estimate is valid for any \(n\ge 2\) and \(1<p<\infty\), whereas the two previous works need to restrict conditions on \(n\) and \(p\). 

After this work was completed, we were informed of a recent excellent work of Beck and Mingione \cite{beck2020lipschitz}. In their paper, they established general theorems on local Lipschitz regularity, especially for solutions to nonuniformly elliptic equations. From some of their estimates \cite[Theorem 1.9 and 1.11]{beck2020lipschitz}, we are able to obtain a more sophisticated estimate than that of Theorem \ref{Lipschitz bound for Crystal Eq}. Although our basic strategy written in Section \ref{Subsect strategy} below seems to be similar to theirs, our individual methods are rather classical, direct and elementary.
Moreover, the details are quite different from theirs.
For more detailed comparison, see Section \ref{Subsect Literature} and Remark \ref{Nonlinear Potential Result}-\ref{case beta=0} in Section \ref{Sect Uniform local Lipschitz bound}.

\subsection{Our strategy}\label{Subsect strategy}
From a viewpoint of comparing \(\Delta_{1}\) with \(\Delta_{p}\) in Section \ref{Subsect typical model}, we describe our strategy briefly.
We first get over the problem how to define the term \(\nabla u/\lvert\nabla u\rvert\) by regarding it as a subdifferential operator. Subdifferential operators often play important roles in many fields of mathematical analysis, including convex analysis \cite{MR2986672}, \cite{MR0274683} and nonlinear semigroup theory \cite{MR2582280}, \cite{MR0348562}. This type of definition has already been taken by Xu \cite{xu2019mathematical} under the Neumann boundary condition.
Our strategy for Theorem \ref{Lipschitz bound for Crystal Eq} is to make suitable approximation schemes, and to avoid analysis on facet. 
Here we illustrate our approaches for local Lipschitz regularity.

For an approximation to (\ref{crystal model eq}), we consider classical solutions to regularized equations
\begin{equation}\label{Regularized Crystal Eq}
-\beta\divx\left(\frac{\nabla u^{\epsilon}}{\sqrt{\epsilon^{2}+\left\lvert\nabla u^{\epsilon}\right\rvert^{2}}}\right)-\divx\left(\left(\epsilon^{2}+\left\lvert\nabla u^{\epsilon}\right\rvert^{2} \right)^{p/2-1}\nabla u^{\epsilon}\right)=f
\end{equation}
for \(0<\epsilon\le 1\). From \cite[Chapter IV and V]{MR0244627}, if \(f\in C^{\infty}(\Omega)\), then for each fixed \(0<\epsilon\le 1\), \(u^{\epsilon}\) admits \(C^{\infty}\)-inner regularity. The key estimate in this paper is the following local a priori estimate;
\begin{equation}\label{a priori estimate}
\sup\limits_{B_{\theta R}}\left\lvert\nabla u^{\epsilon}\right\rvert\le C(n,\,p,\,q,\,\beta,\,\theta)\left(1+\lVert f\rVert_{L^{q}(B_{R})}^{1/(p-1)}+R^{-n/p}\left\lVert\nabla u^{\epsilon}\right\rVert_{L^{p}(B_{R})}\right)
\end{equation}
under the same conditions given in Theorem \ref{Lipschitz bound for Crystal Eq}. From (\ref{a priori estimate}), we first prove Theorem \ref{Lipschitz bound for Crystal Eq} for \(f\in C^{\infty}(\Omega)\). We extend our proof for general \(f\in L^{q}(\Omega)\,(n<q\le\infty)\) by density argument and the H\"{o}lder inequality. To justify this argument, we need some basic properties of solutions to (\ref{crystal model eq}), including the minimizing property of solutions and the stability estimate of solutions. Also, we should make an appropriate justification of convergence \(\nabla u^{\epsilon}\rightarrow\nabla u\) as \(\epsilon\to 0\). Arguments on convergence in the paper are essentially due to Kr\"{u}gel's idea \cite[Theorem 3.3]{krugel2013variational}. More general justification is given in Appendix for the reader's convenience.

The proof of the key estimate (\ref{a priori estimate}) is similar to that of \cite[Proposition 3.3]{MR709038}, but the significant difference is that we have to choose test functions so carefully that their support does not contain any facet of approximated solutions.
We obtain local a priori Lipschitz estimates for solutions to regularized equations in two ways. The first is by Moser's iteration and the second is by De Giorgi's truncation, both of which are used for local boundedness of weak solutions to uniformly elliptic equations. For materials on local boundedness, we refer the reader to \cite[Chapter 3.6]{MR3887613}, \cite[Chapter 8.3, 8.4]{MR3099262}, \cite[Chapter 4.2]{MR2777537} and \cite[Chapter 7.1]{MR2356201}. By testing suitable functions which are supported in a certain regular set of \(\nabla u^{\epsilon}\), we prove local boundedness of \(\nabla u^{\epsilon}\), uniformly for an approximation parameter \(0<\epsilon\le 1\).

Our approaches given above are valid even for generalized equations (\ref{1+p Laplacian}), if \(E_{p}\colon {\mathbb R}^{n}\rightarrow {\mathbb R}\) admits some reasonable properties. These will be stated in Section \ref{Subsect Main theorem} below.

It is still left open whether solutions to (\ref{crystal model eq}) are always \(C^{1,\,\alpha}\) for \(f\in C^{\infty}(\Omega)\), or more generally for \(f\in L^{q}(\Omega)\,(n<q\le\infty)\). To solve this problem, we will probably need more analysis over facet.
\subsection{General result}\label{Subsect Main theorem}
Here we describe proper conditions for equations, and state our main theorem, which covers Theorem \ref{Lipschitz bound for Crystal Eq}.

For regularities, we only require \(f\in L^{q}(\Omega)\, (n< q\le \infty)\) and \(E_{p}\in C^{1}({\mathbb R}^{n})\). However, we also assume that \(E_{p}\) is strictly convex and admits a family of strictly convex functions \(\left\{E_{p}^{\epsilon}\right\}_{0<\epsilon\le 1}\subset C^{\infty}\left({\mathbb R}^{n}\right)\), and that there exists constants \(0<c_{1}\le c_{2}<\infty\), independent of \(0<\epsilon\le 1\), such that
\begin{equation}\label{elliptic p-regular}
c_{1}\lvert z_{0}\rvert^{p}\le E_{p}(z_{0})\le c_{2}\lvert z_{0}\rvert^{p},
\end{equation}
\begin{equation}\label{elliptic condition for relaxed p-th growth term}
c_{1}\left(\epsilon^{2}+\lvert z_{0}\rvert^{2}\right)^{p/2-1}\lvert\zeta\rvert^{2}\le \left\langle\nabla_{z}^{2}E^{\epsilon}(z_{0})\zeta\mathrel{}\middle|\mathrel{}\zeta\right\rangle,
\end{equation}
\begin{equation}\label{bound condition for relaxed p-th growth term}
\left\lvert\left\langle\nabla_{z}^{2}E^{\epsilon}(z_{0})\zeta\mathrel{}\middle| \mathrel{}\omega\right\rangle\right\rvert\le c_{2}\left(\epsilon^{2}+\lvert z_{0}\rvert^{2}\right)^{p/2-1}\lvert\zeta\rvert\lvert\omega\rvert,
\end{equation}
\begin{equation}\label{convergence condition of p-th growth term}
E_{p}(z_{0})\le E_{p}^{\epsilon}(z_{0}),\quad E_{p}(z_{0})=\lim\limits_{\epsilon\to 0} E_{p}^{\epsilon}(z_{0}),
\end{equation}
\begin{equation}\label{convergence condition of nabla p-th growth term}
\nabla_{z} E_{p}(z_{0})=\lim\limits_{\epsilon\to 0} \nabla_{z}E_{p}^{\epsilon}(z_{0})
\end{equation}
for all \(z_{0},\,\zeta,\omega\in {\mathbb R}^{n}\). 
Here \(\langle\,\cdot\mid\cdot\,\rangle\) denotes the canonical inner product in \({\mathbb R}^{n}\). For a sufficiently smooth functional \(E\colon {\mathbb R}^{n}\rightarrow {\mathbb R}\), we also write \(\nabla_{z} E(z_{0})\) and \(\nabla_{z}^{2}E(z_{0})\) as the gradient and the Hessian matrix at \(z_{0}\in{\mathbb R}^{n}\) in classical sense respectively.

A typical example is 
\begin{equation}\label{p-laplacian case}
E_{p}(z)\coloneqq \frac{1}{p}\lvert z\rvert^{p} \textrm{ and } E_{p}^{\epsilon}(z)\coloneqq \frac{1}{p}\left(\epsilon^{2}+\lvert z\rvert^{2}\right)^{p/2}\,(0<\epsilon\le 1).
\end{equation}
It is easy to check that they satisfy \((\ref{elliptic p-regular})\)-\((\ref{convergence condition of nabla p-th growth term})\) with \(c_{1}\coloneqq\min\{\, p-1,\,1/p\,\},\,c_{2}\coloneqq\max\{\,p-1,\, 1\,\}\). For the special case (\ref{p-laplacian case}), \((\ref{1+p Laplacian})\) becomes (\ref{crystal model eq}).

The strategy described in Section \ref{Subsect strategy} yields main theorem in the paper, which states local Lipschitz regularity of solutions to (\ref{1+p Laplacian}).
\begin{theorem}\label{Lipschitz bound for 1+p Laplacian}
Let \(u\) be a solution to (\ref{1+p Laplacian}) in weak sense. Then we have 
\begin{equation}\label{main local Lipschitz estimate}
\lVert\nabla u\rVert_{L^{\infty}(B_{\theta R})}\le C(n,\,p,\,q,\,\beta,\,c_{1},\,c_{2},\,\theta)\left(1+\lVert f\rVert_{L^{q}(B_{R})}^{1/(p-1)}+R^{-n/p}\lVert\nabla u\rVert_{L^{p}(B_{R})}\right)
\end{equation}
for any fixed closed ball \(B_{R}\subset \Omega\) with its radius \(0<R\le 1\), any \(2\le n<q\le\infty\) and \(0<\theta<1\). Here \(0<c_{1}\le c_{2}<\infty\) are constants satisfying (\ref{elliptic p-regular})-(\ref{bound condition for relaxed p-th growth term}).
\end{theorem}
Clearly Theorem \ref{Lipschitz bound for 1+p Laplacian} covers Theorem \ref{Lipschitz bound for Crystal Eq}.

\subsection{Literature overview}\label{Subsect Literature}
Here we describe previous relevant researches, especially on regularities for solutions to (\ref{crystal model eq}), in short.

Elliptic regularity of \(p\)-Laplacian, especially \(C^{1,\,\alpha}\)-regularity of \(p\)-harmonic functions, has been proved by many excellent mathematicians. As a series of papers, we refer the reader to, for instance, Uhlenbeck \cite{MR474389} and Evans \cite{MR672713} for \(2\le p<\infty\) and DiBenedetto \cite{MR709038}, Tolksdorff \cite{MR727034} and Wang \cite{MR1264526} for \(1<p<\infty\). Among them the most related work is one by DiBenedetto \cite{MR709038} in 1983. There he discussed \(C^{1,\,\alpha}\)-regularity of solutions to equations, including
\[-\divx E_{p}(\nabla u)=0\quad\textrm{ in }\Omega.\]
In \cite[Proposition 3.3]{MR709038}, he showed local a priori gradient bounds for solutions to certain regularized equations \[-\divx\nabla_{z}E_{p}^{\epsilon}\left(\nabla u^{\epsilon}\right)=0,\] uniformly for \(0<\epsilon\le 1\). Our proofs of local a priori gradient bounds in Section \ref{Sect Uniform local Lipschitz bound} are essentially obtained by a modification of his arguments. The difference is that we have to make analysis only for regular points, whereas DiBenedetto did make analysis for both degenerate and regular points.

Some mathematical properties of the equation (\ref{crystal model eq}) with \(f=\const\) were discussed in Kr\"{u}gel's thesis in 2013 \cite{krugel2013variational}. 
On local Lipschitz regularity, inspired by the paper \cite{MR1925022}, Kr\"{u}gel proved a local a priori uniform Lipschitz estimate for regularized equations (\ref{Regularized Crystal Eq}) by Moser's iteration \cite{MR170091}. Despite Kr\"{u}gel's claim that the estimate is valid for any \(n\ge 2\) and \(1<p<\infty\), it seems that there need more arguments or modifications especially for \(1<p<2\) (for details, see Rematk \ref{How to modify Kruegel's proof} in Section \ref{Subsect Moser's iteration}). Also, the nonhomogeneous term \(f=\const\) is controlled by an \(L^{\infty}\)-datum in the proof. Our first proof of a local a priori Lipschitz bound (Proposition \ref{Lipschitz Bound}) is similar to \cite[Lemma 4.9]{krugel2013variational}, but our proof works for general \(1<p<\infty\) and \(n<q\le\infty\).
A justification of convergence for approximation schemes was also discussed in the thesis, the results of which are organized more generally in Appendix of this paper.

Recently in 2019, Xu \cite{xu2019mathematical} studied a homogeneous Neumann boundary value problem for a certain nonlinear fourth order equation. There he showed a local Lipschitz estimate for solutions to equations of the type
\begin{equation}\label{Xu eq}
-\beta\divx\left(\frac{\nabla u^{\epsilon}}{\sqrt{\epsilon^{2}+\left\lvert\nabla u^{\epsilon}\right\rvert^{2}}}\right)-\divx\left(\left(\epsilon^{2}+\left\lvert\nabla u^{\epsilon}\right\rvert^{2} \right)^{p/2-1}\nabla u^{\epsilon}\right)=f^{\epsilon}\quad \textrm{ with }\quad \sup\limits_{0<\epsilon\le 1}\left\lVert f^{\epsilon}\right\rVert_{L^{\infty}}<\infty
\end{equation}
by De Giorgi's levelset argument \cite{MR0093649} and analysis on regular points. From this he proved that there exists a solution to the nonlinear fourth order Neumann problem with global Lipschitz continuity under some suitable conditions.
In the proof of uniform Lipschitz bounds for solutions to (\ref{Xu eq}) by Xu, the condition \(n=2\) cannot be removed. This is basically due to the fact that his argument is an adaptation of those given in \cite[Chapter 12.2]{MR1814364}, where elliptic equations in two variables are especially treated. His proof also requires another condition \(p>4/3\) for technical reasons related to estimates for levelsets, and arguments for \(2\le p<\infty\) are almost omitted.
On local a priori Lipschitz bounds for classical solutions, our two proofs are totally different from that given by Xu \cite[Claim 4.1]{xu2019mathematical}. In the first place, the weak formulation (\ref{Euler-Lagrange 1}) in this paper is different from the one used in his paper. While most of Xu's computations are valid only for \(n=2\), our proofs of a priori estimates are valid for general \(n\ge 2\).

On local Lipschitz regularities, our proofs of local a priori estimates given in Section \ref{Sect Uniform local Lipschitz bound} are more general than those from two previous researches by Kr\"{u}gel and Xu, in the sense that our methods are valid for any \(1<p<\infty,\,n\ge 2\) and that the nonhomogeneous term \(f\) is controlled by an \(L^{q}\)-datum with \(n<q\le\infty\). This advantage directly yields our main result of local gradient bounds (Theorem \ref{Lipschitz bound for Crystal Eq}-\ref{Lipschitz bound for 1+p Laplacian}) for any \(n\ge 2,\,1<p<\infty,\,n<q\le\infty\). It is remarkable that the condition \(n<q\le\infty\) is optimal for Lipschitz regularity (see \cite[Section 3]{MR1100802}).

As mentioned in Section \ref{Sect Intro}, a recent paper \cite{beck2020lipschitz} gives us more general results on local Lipschitz regularity for minimizers of variational integrals, especially nonuniformly elliptic ones. These Lipschitz bounds are proved by sophisticated estimates from nonlinear potential theory.
Remarkably, in \cite[Section 1.3]{beck2020lipschitz}, the external term \(f\) is assumed to be only in a Lorentz space \(L(n,\,1)\)
\[\textrm{i.e.,}\quad \lVert f\rVert_{L(n,\,1)(\Omega)}\coloneqq \int_{0}^{\infty}{\mathcal L}^{n}\left(\{x\in\Omega\mid \lvert f(x)\rvert>\lambda\}\right)^{1/n}\,d\lambda <\infty\]
for the case \(n\ge 3\), and for the case \(n=2\) only in an Orlicz space \(\mathrm{L}^{2}(\mathrm{Log}\,\mathrm{L})^{\alpha}\,(\alpha>2)\),
\[\textrm{i.e.,}\quad \int_{\Omega}\lvert f\rvert^{2}\log^{\alpha}(1+\lvert f\rvert)\,dx<\infty\]
for some \(\alpha>2\).
As a special case of \cite[Theorem 1.9]{beck2020lipschitz}, we are able to conclude that 
\begin{equation}\label{Nonlinear potential result}
\sup\limits_{B_{R/2}}\left\lvert\nabla u^{\epsilon}\right\rvert\le C(n,\,p,\,\beta)\left(1+\lVert f\rVert_{L(n,\,1)(B_{R})}^{1/(p-1)}+R^{-n/p}\left\lVert\nabla u^{\epsilon}\right\rVert_{L^{p}(B_{R})}\right)\quad \textrm{for any }n\ge 3 \textrm{ and } B_{R}\subset\Omega,
\end{equation}
where \(u^{\epsilon}\in W^{1,\,p}(\Omega)\) is a weak solution to (\ref{Regularized Crystal Eq}) with \(f\in L(n,\,1)(\Omega)\) (for details, see Remark \ref{Nonlinear Potential Result} in Section \ref{Sect Uniform local Lipschitz bound}).
We recall that continuous and strict inclusions \(L^{n+\epsilon}\subsetneq L(n,\,1)\subsetneq L^{n}\) hold true for any \(\epsilon>0\), and the assumption \(f\in L(n,\,1)\) can be regarded as critical from previous researches on elliptic regularity for solutions to \(-\Delta_{p}u=f\) (see \cite{MR1189042}, \cite{MR607898}).
A sharp estimate for \(n=2,\,f\in \mathrm{L}^{2}(\mathrm{Log}\,\mathrm{L})^{\alpha}\,(\alpha>2)\) can also be deduced from \cite[Theorem 1.11]{beck2020lipschitz} (we note that continuous and strict inclusions \(L^{2+\epsilon}\subsetneq\mathrm{L}^{2}(\mathrm{Log}\,\mathrm{L})^{\alpha}\subsetneq L^{2}\) hold true for any \(\alpha>0\) and \(\epsilon>0\)).

Their strategy for the proof of local Lipschitz bounds \cite[Theorem 1.9 and 1.11]{beck2020lipschitz} broadly consist four parts; construction of approximation schemes \cite[Section 4.1]{beck2020lipschitz}, a Caccioppoli-type estimate for approximated solutions, an iteration \cite[Section 3.1, 4.2 and 4.3]{beck2020lipschitz}, and justification of the convergence \cite[Section 4.4]{beck2020lipschitz}.
It seems that our basic strategy is almost similar to theirs, but in fact the details and individual methods of our proofs are quite different from theirs.
Although our Lipschitz bounds (Theorem \ref{Lipschitz bound for Crystal Eq}-\ref{Lipschitz bound for 1+p Laplacian}, Proposition \ref{Lipschitz Bound}-\ref{De Giorgi's truncation proof}) are somewhat weaker than these sharp estimates by Beck and Mingione, our methods are rather elementary and do not appeal to the nonlinear potential theory \cite{MR1207810} at all. The significant difference is that, compared with a key estimate obtained by a nonlinear iteration argument \cite[Lemma 3.1]{beck2020lipschitz}, our iteration arguments in the proofs of Proposition \ref{Lipschitz Bound}-\ref{De Giorgi's truncation proof} are rather classical and elementary. It should also be noted that another key estimate by Beck and Mingione lies in a Caccioppoli-type estimate \cite[Lemma 4.5]{beck2020lipschitz}, and this is deduced from an weak formulation, which is almost similar to (\ref{tested equation 1}) in this paper. In the proof of \cite[Lemma 4.5]{beck2020lipschitz}, they did fully use De Giorgi's truncation but they did not use Moser's iteration at all, whereas our key estimates in Proposition \ref{Lipschitz Bound} are obtained by Moser's iteration. They chose test functions which differ from those in our proof of Proposition \ref{Lipschitz Bound}-\ref{De Giorgi's truncation proof}, so that our arguments given in Section \ref{Subsect Preliminary} are not needed.
It is sure that they did both make use of approximation schemes and justify the convergence of approximated solutions, but their approaches concerning these are quite different from our direct and elementary ones given in Section \ref{Sect Approximation schemes} and Appendix.

\subsection{Organization of the paper}\label{Subsect Organization}
We outline the contents of the paper.

Section \ref{Sect weak solution} provides a proper definition of weak solutions to (\ref{1+p Laplacian}) in Definition \ref{Def of Weak solutions}. We also prove two properties of weak solutions, the minimizing property of weak solutions (Corollary \ref{solution is minimizer}) and the stability of weak solutions (Corollary \ref{stability of solutions}). These two results are used later in Section \ref{Sect Approximation schemes} to complete the proof of main theorem.

Section \ref{Sect Approximation schemes} deals with approximation schemes. We introduce a parameter \(0<\epsilon\le 1\) and give suitable approximation schemes globally or locally. This approximation argument is inspired by DiBenedetto's work in 1983 \cite{MR709038} and Kr\"{u}gel's doctorial thesis in 2013 \cite{krugel2013variational}.
A justification for convergence is partially discussed by Kr\"{u}gel for some special cases. It is easy to modify arguments therein for general conditions. Results on convergence are used without proof in Section \ref{Sect Approximation schemes}, and the precise proof of these is described in Lemma \ref{strong convergence lemma} in Appendix.
In Section \ref{Subsect Global approximation}, via global approximation we prove Proposition \ref{smooth approximation of subgradient vector field}, which states the converse of Corollary \ref{solution is minimizer}. In Section \ref{Subsect Local approximation}, we give a proof of Theorem \ref{Lipschitz bound for 1+p Laplacian} through local approximation, making use of Lemma \ref{strong convergence lemma}-\ref{weak+uniform bound} in Appendix, Corollary \ref{solution is minimizer}-\ref{stability of solutions} in Section \ref{Sect weak solution}, and Proposition \ref{smooth approximation of subgradient vector field}-\ref{Lipschitz Bound} in Section \ref{Sect Approximation schemes}. Proposition \ref{Lipschitz Bound} in Section \ref{Subsect Local approximation} states a local a priori Lipschitz estimate for solutions to regularized equations, uniformly for an approximation parameter \(0<\epsilon\le 1\) and this plays an important role in the proof of Theorem \ref{Lipschitz bound for 1+p Laplacian}. Proposition \ref{Lipschitz Bound} will be proved in Section \ref{Subsect Moser's iteration}.

Section \ref{Sect Uniform local Lipschitz bound} establishes local a priori Lipschitz estimates for solutions to regularized equations, uniformly for \(0<\epsilon\le 1\). Section \ref{Subsect Preliminary} presents some preliminaries for proofs of local a priori uniform Lipschitz estimates. In Section \ref{Subsect Moser's iteration}, we give a proof of Proposition \ref{Lipschitz Bound} by Moser's iteration. This proof is essentially a modification of that of \cite[Proposition 3.3]{MR709038}, and more general than that of \cite[Lemma 4.9]{krugel2013variational}. In Section \ref{Subsect De Giorgi's truncation}, we also obtain another local a priori uniform Lipschitz estimate by De Giorgi's truncation (Proposition \ref{De Giorgi's truncation proof}). This is an adaptation of the proof of \cite[Theorem 4.1, Method 1]{MR2777537}.

Appendix contains precise proofs of three lemmas (Lemma \ref{Vector inequalities}-\ref{weak+uniform bound}), which are used throughout the paper.

\section{Definition and basic properties of weak solutions}\label{Sect weak solution}
In Section \ref{Sect weak solution}, we define weak solutions to (\ref{1+p Laplacian}).
A proper meaning of \(\nabla u/\lvert\nabla u\rvert\) is given in the sense of a subdifferential.
\begin{definition}\label{Def of Weak solutions}
A pair \((u,\,Z)\in W^{1,\,p}(\Omega)\times L^{\infty}\left(\Omega,\,{\mathbb R}^{n}\right)\) is called a \textit{weak} solution to (\ref{1+p Laplacian}) when it satisfies
\begin{equation}\label{Variational Equarity}
\beta\int_{\Omega}\langle Z\mid\nabla\phi\rangle\,dx+\int_{\Omega}\left\langle\nabla_{z}E_{p}(\nabla u) \mathrel{}\middle|\mathrel{} \nabla\phi\right\rangle\,dx=\int_{\Omega}f\phi\,dx 
\end{equation}
for all \(\phi\in W_{0}^{1,\,p}(\Omega)\), and
\begin{equation}\label{Subgradient}
Z(x)\in\partial\Psi(\nabla u(x))
\end{equation}
for a.e. \(x\in\Omega\). Here \(\partial\Psi(z_{0})\subset{\mathbb R}^{n}\) denotes the subdifferential at \(z_{0}\in{\mathbb R}^{n}\) for the convex functional in \({\mathbb R}^{n}\), \(\Psi(z)\coloneqq \lvert z\rvert\),
\[\textrm{i.e., }\partial\Psi(z_{0})=\left\{\begin{array}{cc}
\left\{\displaystyle\frac{z_{0}}{\lvert z_{0}\rvert} \right\}& (z_{0}\not= 0),\\
\left\{w\in{\mathbb R}^{n}\mathrel{}\middle| \mathrel{}\lvert w\rvert\le 1\right\}& (z_{0}=0).
\end{array}\right.\]
For \(u \in W^{1,p}(\Omega)\), if there is \(Z\in L^{\infty}\left(\Omega,\,{\mathbb R}^{n}\right)\) such that \((u,Z)\) is a weak solution to (\ref{1+p Laplacian}), we simply say that \(u\) is a solution to (\ref{1+p Laplacian}) in \textit{weak} sense.
\end{definition}
\begin{remark}\upshape
To define a weak solution to (\ref{1+p Laplacian}), we may weaken the assumption \(n<q\le\infty\). For example, if \(1<p<n\), then equation (\ref{Variational Equarity}) makes sense for \[(p_{\ast})^{\prime}=\frac{np}{np-n+p}\le q\le\infty,\]
since the Sobolev embedding \(W_{0}^{1,\,p}(\Omega)\hookrightarrow L^{q^{\prime}}(\Omega)\) holds true. We also note that if \(q>(p_{\ast})^{\prime}\), this embedding is compact. Similarly, for the proofs of Corollary \ref{solution is minimizer}-\ref{stability of solutions}, Lemma \ref{lowersemicontinuity of F}, \ref{strong convergence lemma} and Proposition \ref{smooth approximation of subgradient vector field}, it is possible to weaken the assumption \(n<q\le\infty\). We omit this, however, since the assumption \(n<q\le\infty\) is optimal for Lipschitz regularity. Throughout the paper we use the fact that, for a bounded Lipschitz domain \(V\subset{\mathbb R}^{n}\), continous embeddings
\[W_{0}^{1,\,p}(V)\hookrightarrow L^{q^{\prime}}(V),\quad W^{1,\,p}(V)\hookrightarrow L^{q^{\prime}}(V)\]
hold true and they are compact if \(n<q\le\infty\). See \cite[Chapter 4 and 6]{MR2424078} for the complete bibliography.
\end{remark}
\begin{remark}\upshape
Local H\"{o}lder regularity of weak solutions to (\ref{crystal model eq}) can be easily obtained by perturbations from \(p\)-harmonic functions. More regularity property of vector field \(Z\) (for instance, H\"{o}lder regularity) is not discovered yet, which makes it difficult to obtain even local Lipschitz regularity for solutions to (\ref{crystal model eq}). We refer to \cite[Section 2 and 3]{MR1100802} as a related item.
\end{remark}

Before showing basic properties for weak solutions to (\ref{1+p Laplacian}), here we state some elementary estimates on \(E_{p},\,E_{p}^{\epsilon}\).
From (\ref{elliptic p-regular}), it is easy to get
\begin{equation}\label{values of Ep and nabla Ep at 0}
E_{p}(0)=0,\,\nabla_{z}E_{p}(0)=0.
\end{equation}
Therefore we may take sufficiently small \(\epsilon_{0}\in(0,\,1)\) such that
\begin{equation}\label{uniformly bound estimate approximated Ep and nabla Ep at 0}
\sup\limits_{0<\epsilon\le\epsilon_{0}}\left\lvert E_{p}^{\epsilon}(0)\right\rvert\le 1\textrm{ and }\sup\limits_{0<\epsilon\le\epsilon_{0}}\left\lvert\nabla_{z} E_{p}^{\epsilon}(0)\right\rvert\le 1.
\end{equation}
From (\ref{elliptic condition for relaxed p-th growth term})-(\ref{convergence condition of nabla p-th growth term}) and (\ref{values of Ep and nabla Ep at 0}), elementary calculation yields that
\begin{equation}\label{vectror inequalities on coercivity}
\left\langle \nabla_{z}E_{p}(z_{2})-\nabla_{z}E_{p}(z_{1})\mathrel{}\middle|\mathrel{} z_{2}-z_{1}\right\rangle\ge\left\{\begin{array}{cc}
c_{1}\cdot C(p)\lvert z_{1}-z_{2}\rvert^{p} & (p\ge 2),\\
c_{1}\lvert z_{1}-z_{2}\rvert^{2}\left(\epsilon^{2}+\lvert z_{1}\rvert^{2}+\lvert z_{2}\rvert^{2}\right)^{p/2-1} & (1<p<2),
\end{array} \right.
\end{equation}
\begin{equation}\label{uniform continuity of nabla approximated Ep}
\left\lvert \nabla_{z}E_{p}^{\epsilon}(z_{1})-\nabla_{z}E_{p}^{\epsilon}(z_{2}) \right\rvert\le\left\{\begin{array}{cc}
c_{2}\cdot C(p)\left(\epsilon^{p-2}+\lvert z_{1}\rvert^{p-2}+\lvert z_{2}\rvert^{p-2}\right)\lvert z_{1}-z_{2}\rvert & (2\le p<\infty),\\
c_{2}\cdot C(p)\lvert z_{1}-z_{2}\rvert^{p-1}& (1<p<2),
\end{array} \right.
\end{equation}
\begin{equation}\label{growth estimate of nabla Ep}
\left\lvert\nabla_{z} E_{p}(z_{0})\right\rvert\le c_{2}\cdot C(p)\lvert z_{0}\rvert^{p-1},
\end{equation}
\begin{equation}\label{growth estimate of nabla aproximated Ep}
\left\lvert\nabla_{z} E_{p}^{\epsilon}(z_{0})-\nabla_{z} E_{p}^{\epsilon}(0)\right\rvert\le \left\{\begin{array}{cc}
c_{2}\cdot C(p)\left(\epsilon^{p-1}+\lvert z_{0}\rvert^{p-1}\right) & (2\le p<\infty),\\
c_{2}\cdot C(p)\lvert z_{0}\rvert^{p-1}& (1<p<2),
\end{array} \right.
\end{equation}
\begin{equation}\label{growth estimate of approximated Ep}
\left\lvert E_{p}^{\epsilon}(z_{0})-E_{p}^{\epsilon}(0)\right\rvert\le\left\{\begin{array}{cc}
C(c_{2},\,p)\left(\epsilon^{p-1}\lvert z_{0}\rvert+\left\lvert \nabla_{z}E_{p}^{\epsilon}(0)\right\rvert\lvert z_{0}\rvert+\lvert z_{0}\rvert^{p} \right) & (2\le p<\infty),\\
C(c_{2},\,p)\left(\left\lvert \nabla_{z}E_{p}^{\epsilon}(0)\right\rvert\lvert z_{0}\rvert+\lvert z_{0}\rvert^{p} \right) & (1<p<2),
\end{array}\right.
\end{equation}
\begin{align}\label{coercivity of approximated Ep}
E_{p}^{\epsilon}(z_{0})-E_{p}^{\epsilon}(0)-\left\langle\nabla_{z}E_{p}^{\epsilon}(0) \mathrel{}\middle|\mathrel{} z_{0}\right\rangle&
\ge\left\langle \nabla_{z}E_{p}^{\epsilon}(z_{0})-\nabla_{z}E_{p}^{\epsilon}(0)\mathrel{}\middle|\mathrel{} z_{0}\right\rangle\nonumber\\&\ge\left\{\begin{array}{cc}
c_{1}\cdot C(p)\lvert z_{0}\rvert^{p} & (2\le p<\infty),\\
c_{1}\left[\left(\epsilon^{2}+\lvert z_{0}\rvert^{2}\right)^{p/2}-\epsilon^{p}\right] & (1<p<2),
\end{array}\right.
\end{align}
for all \(z_{0},\,z_{1},\,z_{2}\in{\mathbb R}^{n}\) and \(0<\epsilon\le 1\). Here we omit the proof of (\ref{vectror inequalities on coercivity})-(\ref{coercivity of approximated Ep}). For details, see Lemma \ref{Vector inequalities} in Appendix.
\begin{remark}\upshape
We can deduce an inequality of the type (\ref{elliptic p-regular}) from (\ref{convergence condition of p-th growth term})-(\ref{convergence condition of nabla p-th growth term}) and (\ref{growth estimate of approximated Ep})-(\ref{coercivity of approximated Ep}).
Therefore we may assume (\ref{elliptic condition for relaxed p-th growth term})-(\ref{convergence condition of nabla p-th growth term}) and (\ref{values of Ep and nabla Ep at 0}), instead of (\ref{elliptic p-regular})-(\ref{convergence condition of nabla p-th growth term}).
\end{remark}

As pointed out in Section \ref{Subsect typical model}, equation (\ref{1+p Laplacian}) derives from a minimizing problem of variational integral
\begin{equation}\label{Functional involving linear growth}
F_{\Omega}(u)\coloneqq \beta\int_{\Omega}\lvert\nabla u\rvert\,dx+\int_{\Omega}E_{p}(\nabla u)\,dx-\int_{\Omega}fu\,dx
\end{equation}
under a certain boundary condition.
We first verify that a weak solution to (\ref{1+p Laplacian}) is a minimizer of the functional \(F_{\Omega}\) on a suitable function class.
\begin{corollary}\label{solution is minimizer}
Let \((u,\,Z)\in W^{1,\,p}(\Omega)\times L^{\infty}\left(\Omega,\,{\mathbb R}^{n}\right)\) be a weak solution to (\ref{1+p Laplacian}). Then we obtain \(F_{\Omega}(u)\le F_{\Omega}(v)\) for all \(v\in u+W_{0}^{1,\,p}(\Omega)\). Here \(F_{\Omega}\colon W^{1,\,p}(\Omega)\rightarrow {\mathbb R}\) is defined as in (\ref{Functional involving linear growth}).
\end{corollary}
\begin{proof}
We note that \(\partial E_{p}(z_{0})=\left\{\nabla_{z}E_{p}(z_{0})\right\}\) for all \(z_{0}\in{\mathbb R}^{n}\), since \(E_{p}\in C^{1}\left({\mathbb R}^{n}\right)\) is convex. Combining this with (\ref{Subgradient}), we have subgradient inequalities
\[\lvert \nabla v\rvert -\lvert\nabla u\rvert\ge \langle Z\mid\nabla (v-u)\rangle,\quad E_{p}(\nabla v)-E_{p}(\nabla u)\ge \left\langle\nabla_{z}E_{p}(\nabla u) \mathrel{}\middle|\mathrel{} \nabla(v-u)\right\rangle\quad \textrm{ a.e. in } \Omega.\]
Testing \(\phi\coloneqq v-u\in W_{0}^{1,\,p}(\Omega)\) in (\ref{Variational Equarity}), we obtain 
\[0=\beta\int_{\Omega}\langle Z\mid\nabla(v-u)\rangle\,dx+\int_{\Omega}\left\langle\nabla_{z}E_{p}(\nabla u) \mathrel{}\middle|\mathrel{} \nabla(v-u)\right\rangle\,dx-\int_{\Omega}f(v-u)\,dx \le F_{\Omega}(v)-F_{\Omega}(u). \qedhere\]
\end{proof}

We also mention the stability estimate of solutions, which is needed to complete the proof of Theorem \ref{Lipschitz bound for 1+p Laplacian}.

\begin{corollary}\label{stability of solutions}
Let \(f_{1},\,f_{2}\in L^{q}(\Omega)\,(n< q\le\infty)\).
Assume that \((u_{1},\,Z_{1}),\,(u_{2},\,Z_{2})\in W^{1,\,p}(\Omega)\times L^{\infty}(\Omega,\,{\mathbb R}^{n})\) satisfy
\[-\beta \divx (\nabla u_{j}/\lvert \nabla u_{j}\rvert)-\divx\nabla_{z} E_{p}(\nabla u_{j})\ni f_{j}\quad  \textrm{ in } \Omega \quad \textrm{ for each }j\in\{\,1,\,2\,\}\]
in weak sense. If \(u_{1}-u_{2}\in W_{0}^{1,\,p}(\Omega)\), then we obtain
\begin{equation}\label{Stability for p>2}
\lVert \nabla u_{1}-\nabla u_{2}\rVert_{L^{p}(\Omega)}\le C(n,\,\,p,\,q,\,c_{1},\,\Omega)\lVert f_{1}-f_{2}\rVert_{L^{q}(\Omega)}^{1/(p-1)}
\end{equation}
for \(p\ge 2\). For \(1<p<2\), instead of (\ref{Stability for p>2}) we obtain
\begin{equation}\label{Stability for 1<p<2}
\lVert \nabla u_{1}-\nabla u_{2}\rVert_{L^{1}(\Omega)}\le C(n,\,p,\,q,\,\beta,\,c_{1},\,c_{2},\,\Omega)\left(1+\lVert \nabla u_{2}\rVert_{L^{p}(\Omega)}^{p}+\lVert f_{1}\rVert_{L^{q}(\Omega)}^{p^{\prime}}\right) \lVert f_{1}-f_{2}\rVert_{L^{q}(\Omega)}^{1/2},
\end{equation}
where \(p^{\prime}\coloneqq p/(p-1)\in(1,\,\infty)\) denotes the H\"{o}lder conjugate of \(p\).
\end{corollary}
\begin{proof}
Test \(u_{1}-u_{2}\in W_{0}^{1,\,p}(\Omega)\) in each equation. Then we obtain
\[\int_{\Omega}\langle Z_{1}-Z_{2}\mid \nabla(u_{1}-u_{2})\rangle\,dx+\int_{\Omega}\langle \nabla_{z}E_{p}(\nabla u_{1})-\nabla_{z}E_{p}(\nabla u_{2})\mid \nabla(u_{1}-u_{2})\rangle\,dx=\int_{\Omega}(f_{1}-f_{2})(u_{1}-u_{2})\,dx.\]
Since the subdifferential operator \(\partial\Psi=\partial\lvert\,\cdot\,\rvert\) is monotone (see for instance \cite{MR0348562}), we deduce that
\[\langle Z_{1}-Z_{2}\mid \nabla(u_{1}-u_{2})\rangle\ge 0\quad \textrm{ a.e. in }\Omega\]
from (\ref{Subgradient}).
By (\ref{vectror inequalities on coercivity}), we obtain
\[\langle \nabla_{z}E_{p}(\nabla u_{1})-\nabla_{z}E_{p}(\nabla u_{2})\mid \nabla(u_{1}-u_{2})\rangle\ge\left\{\begin{array}{cc}
c_{1}\cdot C(p)\lvert \nabla (u_{1}-u_{2})\rvert^{p}& (p\ge 2),\\
c_{1}(1+\lvert \nabla u_{1}\rvert^{2}+\lvert\nabla u_{2} \rvert^{2})^{p/2-1}\lvert\nabla(u_{1}-u_{2})\rvert^{2}& (1<p<2),
\end{array}\right.\quad \textrm{ a.e. in }\Omega.\]
By the Sobolev embedding \(W_{0}^{1,\,p}(\Omega)\hookrightarrow L^{q^{\prime}}(\Omega)\), we get for \(p\ge 2\),
\begin{align*}
c_{1}\cdot C(p)\lVert\nabla u_{1}-\nabla u_{2}\rVert_{L^{p}(\Omega)}^{p}&\le \int_{\Omega}\lvert f_{1}-f_{2}\rvert \lvert u_{1}-u_{2}\rvert\,dx\\&\le C(n,\,p,\,q,\,\Omega)\lVert f_{1}-f_{2}\rVert_{L^{q}(\Omega)}\lVert\nabla u_{1}-\nabla u_{2} \rVert_{L^{p}(\Omega)}.
\end{align*}
From this we conclude (\ref{Stability for p>2}). Similarly for \(1<p<2\), we get
\begin{align*}
\lVert \nabla u_{1}-\nabla u_{2}\rVert_{L^{1}(\Omega)}&\le \left(\int_{\Omega}\lvert \nabla (u_{1}-u_{2})\rvert^{2}\left(1+\lvert\nabla u_{1}\rvert^{2}+\lvert\nabla u_{2}\rvert^{2}\right)^{p/2-1}\,dx\right)^{1/2}\left(\int_{\Omega}\left(1+\lvert\nabla u_{1}\rvert^{2}+\lvert\nabla u_{2}\rvert^{2}\right)^{1-p/2}\,dx\right)^{1/2}\\&\le C(n,\,p,\,q,\,\Omega)\lVert f_{1}-f_{2}\rVert_{L^{q}(\Omega)}^{1/2}\lVert\nabla (u_{1}-u_{2}) \rVert_{L^{p}(\Omega)}^{1/2}\left(\int_{\Omega}\left(1+\lvert\nabla u_{1}\rvert^{2-p}+\lvert\nabla u_{2}\rvert^{2-p}\right)\,dx\right)^{1/2}\\&\le C(n,\,p,\,q,\,\Omega)\lVert f_{1}-f_{2}\rVert_{L^{q}(\Omega)}^{1/2}\left(1+\lVert\nabla u_{1}\rVert_{L^{p}(\Omega)}^{p}+\lVert\nabla u_{2}\rVert_{L^{p}(\Omega)}^{p} \right)
\end{align*}
by the Young inequaltiy (see \cite[Chapter 7.1]{MR1814364}, \cite[Chapter 2.1 (3)]{MR0244627}). It suffices to show that
\begin{equation}\label{claim for stability estimate}
\lVert \nabla u_{1}\rVert_{L^{p}(\Omega)}^{p}\le C(n,\,p,\,q,\,\beta,\,c_{1},\,c_{2},\,\Omega)\left(1+\lVert\nabla u_{2}\rVert_{L^{p}(\Omega)}^{p}+\lVert f_{1}\rVert_{L^{q}(\Omega)}^{p^{\prime}}\right)
\end{equation}
to complete the proof of (\ref{Stability for 1<p<2}).
By (\ref{elliptic p-regular}), the Young inequality, the H\"{o}lder inequality and the inequality
\[\beta\int_{\Omega}\lvert\nabla u_{1}\rvert\,dx+\int_{\Omega}E_{p}(\nabla u_{1})\,dx-\int_{\Omega}f_{1}u_{1}\,dx\le \beta\int_{\Omega}\lvert\nabla u_{2}\rvert\,dx+\int_{\Omega}E_{p}(\nabla u_{2})\,dx-\int_{\Omega}f_{1}u_{2}\,dx \] from Corollary \ref{solution is minimizer},
we get
\begin{align*}
c_{1}\lVert\nabla u_{1}\rVert_{L^{p}(\Omega)}^{p}+\beta\lVert\nabla u_{1}\rVert_{L^{1}(\Omega)}&\le c_{2}\lVert\nabla u_{2}\rVert_{L^{p}(\Omega)}^{p}+\beta\lVert\nabla u_{2}\rVert_{L^{1}(\Omega)}+C(n,\,p,\,q,\,\Omega)\lVert f_{1}\rVert_{L^{q}(\Omega)}\lVert \nabla (u_{1}-u_{2})\rVert_{L^{p}(\Omega)}\\&\le C(\beta,\,c_{2},\,\Omega)\left( 1+\lVert\nabla u_{2}\rVert_{L^{p}(\Omega)}^{p}\right)+C(n,\,p,\,q,\,\Omega)\lVert f_{1}\rVert_{L^{q}(\Omega)}\lVert \nabla u_{2}\rVert_{L^{p}(\Omega)}\\&\quad +C(n,\,p,\,q,\,\Omega)\lVert f_{1}\rVert_{L^{q}(\Omega)}\lVert \nabla u_{1}\rVert_{L^{p}(\Omega)}\\&\le C(n,\,p,\,q,\,\beta,\,c_{1},\,c_{2},\,\Omega)\left(1+\lVert\nabla u_{2}\rVert_{L^{p}(\Omega)}^{p}+\lVert f_{1}\rVert_{L^{q}(\Omega)}^{p^{\prime}}\right) +\frac{c_{1}}{2}\lVert \nabla u_{1}\rVert_{L^{p}(\Omega)}^{p}.
\end{align*}
From this we conclude (\ref{claim for stability estimate}). 
\end{proof}

\section{Approximation schemes}
\label{Sect Approximation schemes}
For each \(0<\epsilon\le 1\), we consider a weak solution \(u^{\epsilon}\) to the equation
\begin{equation}\label{classical smooth equation}
-\divx\nabla_{z}E^{\epsilon}(\nabla u^{\epsilon})=f\in L^{q}(\Omega)\,(n<q\le\infty)
\end{equation} 
in either \(\Omega\) or Lipschitz subdomain \(U\Subset \Omega\).
Here a family of strictly convex functions \(\left\{E^{\epsilon}\right\}_{0<\epsilon\le 1}\subset C^{\infty}\left({\mathbb R}^{n}\right)\) admits constants \(0<C_{1}\le C_{2}<\infty\), independent of \(0<\epsilon\le 1\), such that
\begin{equation}\label{elliptic condition}
C_{1}\left(\epsilon^{2}+\lvert z\rvert^{2}\right)^{p/2-1}\lvert\zeta\rvert^{2}\le \left\langle\nabla_{z}^{2}E^{\epsilon}(z)\zeta\mathrel{}\middle|\mathrel{}\zeta\right\rangle,
\end{equation}
\begin{equation}\label{bound condition}
\left\lvert\left\langle\nabla_{z}^{2}E^{\epsilon}(z_{0})\zeta\mathrel{}\middle| \mathrel{}\omega\right\rangle\right\rvert\le C_{2}\left(\epsilon^{2}+\lvert z_{0}\rvert^{2}\right)^{p/2-1}\lvert\zeta\rvert\lvert\omega\rvert
\end{equation}
for all \(z_{0},\,\zeta,\omega\in {\mathbb R}^{n}\) with \(\lvert z_{0}\rvert\ge 1\). 
Especially in this paper, we consider
\begin{equation}\label{Example of smooth functional}
\Psi^{\epsilon}(z)\coloneqq \sqrt{\epsilon^{2}+\lvert z\rvert^{2}},\quad E^{\epsilon}(z)\coloneqq \beta\Psi^{\epsilon}(z)+E_{p}^{\epsilon}(z)\quad \textrm{ for }z\in{\mathbb R}^{n},
\end{equation}
where \(E_{p}\) and \(\left\{E_{p}^{\epsilon}\right\}_{0<\epsilon\le 1}\) satisfy (\ref{elliptic p-regular})-(\ref{convergence condition of nabla p-th growth term}).
By direct calculation it is easy to check that eigenvalues of \(\nabla_{z}^{2}\Psi^{\epsilon}(z_{0})\), the Hessian matrix of \(\Psi^{\epsilon}\) at \(z_{0}\in{\mathbb R}^{n}\), are given by \(\beta\left(\epsilon^{2}+\lvert z_{0}\rvert^{2}\right)^{-1/2}\) and \(\beta\epsilon^{2}\left(\epsilon^{2}+\lvert z_{0}\rvert^{2}\right)^{-3/2}\).
Hence \(E^{\epsilon}\) defined as in (\ref{Example of smooth functional}) satisfies (\ref{elliptic condition})-(\ref{bound condition}) with \(C_{1}\coloneqq c_{1},\,C_{2}\coloneqq c_{2}+\beta\).

Equation (\ref{classical smooth equation}) derives from the Euler-Lagrange equation of the regularized variational integral
\[F_{V}^{\epsilon}(u)\coloneqq\int_{V}E^{\epsilon}(\nabla u)\,dx-\int_{V}fu\,dx,\]
where \(V=\Omega\) or \(V=U\Subset \Omega\). We also define a functional \(F_{U}\colon W^{1,\,p}(U)\rightarrow {\mathbb R}\) as in (\ref{Functional involving linear growth}), replacing \(\Omega\) by \(U\).
\begin{lemma}\label{lowersemicontinuity of F}
Functionals \(F_{V},\,F_{V}^{\epsilon}\,(0<\epsilon\le 1)\) are lower semi-continuous in \(W^{1,\,p}(V)\) with respect to the weak topology.
\end{lemma}
Lower semi-continuity of convex energy functionals with respect to the weak topology is generally discussed in \cite[Chapter 8.2.2]{MR1625845} (see also \cite[Chapter I.2]{MR717034}, \cite[Chapter 4.2 and 4.3]{MR1962933}). It is easy to prove Lemma \ref{lowersemicontinuity of F} by making an adaptation of arguments therein. However, we give another simpler proof of Lemma \ref{lowersemicontinuity of F} by showing that functionals \(F_{V},\,F_{V}^{\epsilon}\) are continuous with respect to the strong topology.
\begin{proof}
By \cite[Corollary 3.9]{MR2759829}, we are reduced to showing that convex functionals \(F_{V},\,F_{V}^{\epsilon}\,(0<\epsilon\le 1)\) are continuous in \(W^{1,\,p}(V)\) with respect to the strong topology.
Fix \(0<\epsilon\le 1\) and let \(\{v_{n}\}_{n=1}^{\infty}\subset W^{1,\,p}(V)\) satisfy \(v_{n}\rightarrow v\,(n\to\infty)\) in \(W^{1,\,p}(V)\) for some \(v\in W^{1,\,p}(V)\). We verify that \(F_{V}^{\epsilon}(v_{n})\to F_{V}^{\epsilon}(v)\,(n\to\infty)\). Take any subsequence \(\{v_{n_{j}}\}_{j=1}^{\infty}\subset \{v_{n}\}_{n=1}^{\infty}\). By \cite[Theorem 4.9]{MR2759829} and the continuous embedding \(W^{1,\,p}(V)\hookrightarrow L^{q^{\prime}}(V)\), there exists a subsequence \(\{v_{n_{j_{k}}}\}_{k=1}^{\infty}\subset \{v_{n_{j}}\}_{j=1}^{\infty}\) and \(w\in L^{p}(V)\) such that 
\begin{equation}\label{a.e. conv}
\nabla v_{n_{j_{k}}}\rightarrow \nabla v\,(k\to\infty)\quad\textrm{ a.e. in }V,
\end{equation}
\begin{equation}\label{Lp control}
\lvert v_{n_{j_{k}}}\rvert\le w\quad\textrm{ a.e. in }V \textrm{ and for all } k\in {\mathbb N},
\end{equation}
\begin{equation}\label{Lp' conv}
v_{n_{j_{k}}}\rightarrow v\,(k\to\infty)\quad\textrm{ in }L^{p^{\prime}}(V).
\end{equation}
By (\ref{a.e. conv}) and \(E_{p}^{\epsilon}\in C^{\infty}\left({\mathbb R}^{n}\right)\), we get
\[E^{\epsilon}(\nabla v_{n_{j_{k}}})\rightarrow E^{\epsilon}(\nabla v)\,(k\to\infty)\quad \textrm{ a.e. in }V.\]
By (\ref{growth estimate of approximated Ep}), (\ref{Lp control}) and the Young inequality, we can easily check that
\begin{align*}
E^{\epsilon}(\nabla v_{n_{j_{k}}})&\le \left\lvert E_{p}^{\epsilon}(0)\right\rvert+ C(c_{2},\,p)\left(\epsilon^{p-1}\lvert \nabla v_{n_{j_{k}}}\rvert+\lvert \nabla_{z}E_{p}^{\epsilon}(0)\rvert\lvert \nabla v_{n_{j_{k}}}\rvert+\lvert \nabla v_{n_{j_{k}}}\rvert^{p} \right)\\&\le \left\lvert E_{p}^{\epsilon}(0)\right\rvert+ C(c_{2},\,p)\left(\epsilon^{p}+\lvert \nabla_{z}E_{p}^{\epsilon}(0)\rvert^{p^{\prime}}+w^{p} \right)\in L^{1}(V) \quad \textrm{ a.e. in }V,
\end{align*}
uniformly for \(k\in{\mathbb N}\).
From these we obtain
\[\lim_{k\to\infty} F_{V}^{\epsilon}(v_{n_{j_{k}}})=\lim_{k\to\infty}\int_{V}E^{\epsilon}(\nabla v_{n_{j_{k}}})\,dx-\lim_{k\to\infty}\int_{V}fv_{n_{j_{k}}}\,dx=\int_{V}E^{\epsilon}(\nabla v)\,dx-\int_{V}fv\,dx=F_{V}^{\epsilon}(v)\]
by Lebesgue's dominated convergence theorem and (\ref{Lp' conv}).
Hence it follows that \(F_{V}^{\epsilon}(v_{n})\to F_{V}^{\epsilon}(v)\,(n\to\infty)\). This means that \(F_{V}^{\epsilon}\) is strongly continuous in \(W^{1,\,p}(V)\). From (\ref{elliptic p-regular}) and \(E_{p}\in C^{1}\left({\mathbb R}^{n}\right)\), we similarly conclude that \(F_{V}\) is strongly continuous in \(W^{1,\,p}(V)\). 
\end{proof}

For each fixed \(u_{0}\in W^{1,\,p}(V)\) and \(0<\epsilon\le 1\), we can define
\[u^{\epsilon}\coloneqq \argmin\left\{F_{V}^{\epsilon}(v)\mathrel{}\middle|\mathrel{}v\in u_{0}+W_{0}^{1,\,p}(V)\right\}\in u_{0}+W_{0}^{1,\,p}(V).\]
Using the Young inequality, we can easily check that for all \(v\in u_{0}+W_{0}^{1,\,p}(V)\),
\begin{align*}
F_{V}^{\epsilon}(v)&\ge \int_{V}E_{p}^{\epsilon}(0)\,dx+\int_{V}\left\langle\nabla_{z}E_{p}^{\epsilon}(0)\mathrel{}\middle|\mathrel{}\nabla v \right\rangle\,dx+ C(c_{1},\,p)\int_{V}\left(\lvert \nabla v\rvert^{p}-1\right)\,dx\\&\quad-\lVert f\rVert_{L^{q}(V)}\lVert v-u_{0}\rVert_{L^{q^{\prime}}(V)}-\lVert f\rVert_{L^{q}(V)}\lVert u_{0}\rVert_{L^{q^{\prime}}(V)}\quad(\textrm{by (\ref{coercivity of approximated Ep}) and the H\"{o}lder inequality}) \\&\ge \frac{C(c_{1},\,p)}{2}\lVert \nabla v\rVert_{L^{p}(V)}^{p}-\lVert f\rVert_{L^{q}(V)}\lVert \nabla(v-u_{0})\rVert_{L^{p}(V)}\\&-C\left(n,\,p,\,q,\,c_{1},\,V,\,E_{p}^{\epsilon}(0),\,\nabla_{z}E_{p}^{\epsilon}(0)\right)\left(1+\lVert f\rVert_{L^{q}(V)}\lVert u_{0}\rVert_{W^{1,\,p}(V)}\right)\\&\quad\left(\textrm{by the Sobolev embedding }W^{1,\,p}(V),\,W_{0}^{1,\,p}(V)\hookrightarrow L^{q^{\prime}}(V)\right) \\&\ge \frac{C(c_{1},\,p)}{4}\lVert \nabla v\rVert_{L^{p}(V)}^{p}-C\left(n,\,p,\,q,\,c_{1},\,V,\,E_{p}^{\epsilon}(0),\,\nabla_{z}E_{p}^{\epsilon}(0)\right)\left(1+\lVert f\rVert_{L^{q}(V)}\lVert u_{0}\rVert_{W^{1,\,p}(V)}+\lVert f\rVert_{L^{q}(V)}^{p^{\prime}} \right).
\end{align*}
Combining this with Lemma \ref{lowersemicontinuity of F}, we conclude that \(F_{U}^{\epsilon}\) is coercive and weakly lower semi-continuous in \(u_{0}+W_{0}^{1,\,p}(V)\).
Hence the existence of a minimizer \(u^{\epsilon}\in u_{0}+W_{0}^{1,\,p}(V)\) is guaranteed by direct method (see for instance \cite[Chapter 8.2.2]{MR1625845}, \cite[Chapter I.3 and I.4]{MR717034}, \cite[Chapter 4.4]{MR1962933}). Uniqueness is clear by strict convexity of \(F_{V}^{\epsilon}\) in \(u_{0}+W^{1,\,p}(V)\), since \(E_{p}^{\epsilon}\) is strictly convex.
Similarly we can determine a unique function
\[u\coloneqq \argmin\left\{F_{V}(v)\mathrel{}\middle|\mathrel{}v\in u_{0}+W_{0}^{1,\,p}(V)\right\}\in u_{0}+W_{0}^{1,\,p}(V)\]
for each \(u_{0}\in W^{1,\,p}(V)\). We note that it is easy to deduce that \(F_{V}\) is coercive in \(u_{0}+W_{0}^{1,\,p}\) from (\ref{elliptic p-regular}). Lemma \ref{strong convergence lemma} in Appendix states that \(u^{\epsilon}\rightarrow u\) in \(W^{1,\,p}(V)\) as \(\epsilon\to 0\), up to a subsequence. Results from Lemma \ref{strong convergence lemma} are used throughout Section \ref{Sect Approximation schemes}.

\subsection{Global approximation}
\label{Subsect Global approximation}
From Corollary \ref{solution is minimizer}, if \(u\in W^{1,\,p}(\Omega)\) is a solution to (\ref{1+p Laplacian}) in weak sense, then \(u\) satisfies
\begin{equation}\label{u is minimizer}
u=\argmin\left\{F_{\Omega}(v)\mathrel{}\middle|\mathrel{}v\in u+W_{0}^{1,\,p}(\Omega)\right\}.
\end{equation}
Proposition \ref{smooth approximation of subgradient vector field} states that the converse is true.

\begin{proposition}\label{smooth approximation of subgradient vector field}
Let \(f\in L^{q}(\Omega)\,(n<q\le\infty)\). Assume that \(u\in W^{1,\,p}(\Omega)\) satisfies (\ref{u is minimizer}).
Then \(u\) is a solution to (\ref{1+p Laplacian}) in weak sense. That is, there exists \(Z\in L^{\infty}\left(\Omega,\,{\mathbb R}^{n}\right)\) such that \((u,\,Z)\in W^{1,\,p}(\Omega)\times L^{\infty}\left(\Omega,\,{\mathbb R}^{n}\right)\) is a weak solution to (\ref{1+p Laplacian}).
\end{proposition}
\begin{proof}
For each \(0<\epsilon\le 1\), we set
\[u^{\epsilon}\coloneqq \argmin\left\{F_{\Omega}^{\epsilon}(v)\mathrel{}\middle|\mathrel{} v\in u+W_{0}^{1,\,p}(\Omega)\right\}\in u+W_{0}^{1,\,p}(\Omega).\]
By Lemma \ref{strong convergence lemma} in Appendix, we have \(u^{\epsilon}\rightarrow u\) in \(W^{1,\,p}(\Omega)\), up to a subsequence. We note that
\[\left\lvert\nabla_{z}\Psi^{\epsilon}\left(\nabla u^{\epsilon}\right)\right\rvert=\displaystyle\frac{\lvert\nabla u^{\epsilon}\rvert}{\sqrt{\epsilon^{2}+\lvert\nabla u^{\epsilon}\rvert^{2}}}\le 1\quad \textrm{ a.e. in }\Omega.\]
By \cite[Corollary 3.30 and Theorem 4.9]{MR2759829}, again up to a subsequence, we may assume that
\begin{equation}\label{lp control}
\left\lvert \nabla u^{\epsilon}\right\rvert\le v \quad\textrm{ a.e. in }\Omega \textrm{ and for all } 0<\epsilon\le 1,
\end{equation}
\begin{equation}\label{a.e. convergence}
\nabla u^{\epsilon} \to \nabla u\,(\epsilon\to 0)\quad \textrm{ a.e. in }\Omega,
\end{equation}
\begin{equation}\label{weak* convergence}
\nabla_{z}\Psi^{\epsilon}\left(\nabla u^{\epsilon}\right)=\displaystyle\frac{\nabla u^{\epsilon}}{\sqrt{\epsilon^{2}+\lvert\nabla u^{\epsilon}\rvert^{2}}} \overset{\ast}{\rightharpoonup} Z\quad \textrm{ in }L^{\infty}\left(\Omega,\,{\mathbb R}^{n}\right)
\end{equation}
for some \(Z\in L^{\infty}\left(\Omega,\,{\mathbb R}^{n}\right),\,v\in L^{p}(\Omega)\). (\ref{a.e. convergence})-(\ref{weak* convergence}) imply that
\[\lVert Z\rVert_{L^{\infty}\left(\Omega,\,{\mathbb R}^{n}\right)}\le 1,\quad Z(x)=\displaystyle\frac{\nabla u(x)}{\lvert\nabla u(x)\rvert}\,\textrm{ if }\nabla u(x)\not= 0.\]
Hence \(Z\) satisfies (\ref{Subgradient}). Consider the Euler-Lagrange equation of \(u^{\epsilon}\), then we have
\begin{equation}\label{Approximated Euler-Lagrange eq in weak sense}
\int_{\Omega}\left\langle\nabla_{z}\Psi^{\epsilon}\left(\nabla u^{\epsilon}\right)\mathrel{}\middle|\mathrel{}\nabla\phi\right\rangle\,dx+\int_{\Omega}\left\langle\nabla_{z}E_{p}^{\epsilon}\left(\nabla u^{\epsilon}\right)\mathrel{}\middle|\mathrel{}\nabla\phi\right\rangle\,dx=\int_{\Omega}f\phi\,dx
\end{equation}
for all \(\phi\in W_{0}^{1,\,p}(\Omega)\). We claim that
\begin{equation}\label{Lp' strong convergence}
\nabla_{z}E_{p}^{\epsilon}\left(\nabla u^{\epsilon}\right)\rightarrow \nabla_{z}E_{p}(\nabla u)\quad \textrm{ in }L^{p^{\prime}}\left(\Omega,\,{\mathbb R}^{n}\right).
\end{equation}
From (\ref{convergence condition of nabla p-th growth term}), (\ref{uniformly bound estimate approximated Ep and nabla Ep at 0}), (\ref{uniform continuity of nabla approximated Ep}), (\ref{growth estimate of nabla aproximated Ep}) and the Arzel\`{a}-Ascoli theorem, we conclude that
\[\nabla_{z}E_{p}^{\epsilon}\rightarrow\nabla_{z}E_{p}\,(\epsilon\to 0)\quad \textrm{ compactly in }{\mathbb R}^{n}.\]
Combining this result with (\ref{a.e. convergence}), we get
\[\nabla_{z}E_{p}^{\epsilon}\left(\nabla u^{\epsilon}\right)\rightarrow \nabla_{z}E_{p}(\nabla u)\quad \textrm{ a.e. in }\Omega.\]
By (\ref{values of Ep and nabla Ep at 0}), (\ref{growth estimate of nabla Ep})-(\ref{growth estimate of nabla aproximated Ep}) and (\ref{lp control}), we obtain
\begin{align*}
\left\lvert\nabla_{z}E_{p}^{\epsilon}\left(\nabla u^{\epsilon}\right)-\nabla_{z}E_{p}(\nabla u)\right\rvert&\le \left\lvert \nabla_{z}E_{p}^{\epsilon}(0)\right\rvert+\left\lvert\nabla E_{p}(\nabla u)\right\rvert+\left\lvert\nabla_{z}E_{p}^{\epsilon}\left(\nabla u^{\epsilon}\right)-\nabla_{z}E_{p}^{\epsilon}(0)\right\rvert\\&\le 1+c_{2}\cdot C(p)\lvert\nabla u\rvert^{p-1}+c_{2}\cdot C(p)\left(1+\left\lvert\nabla u^{\epsilon}\right\rvert^{p-1}\right)\\&\le C(c_{2},\,p)\left(1+\lvert\nabla u\rvert^{p-1}+v^{p-1}\right)\in L^{p^{\prime}}(\Omega)\quad \textrm{ a.e. in }\Omega,
\end{align*}
uniformly for \(0<\epsilon\le\epsilon_{0}\).
Hence by Lebesgue's domnated convergence theorem, we conclude (\ref{Lp' strong convergence}).
From (\ref{weak* convergence})-(\ref{Lp' strong convergence}), we easily verify that \((u,\,Z)\in W^{1,\,p}(\Omega)\times L^{\infty}\left(\Omega,\,{\mathbb R}^{n}\right)\) satisfies (\ref{Variational Equarity}) for all \(\phi\in W_{0}^{1,\,p}(\Omega)\), and it completes the proof. 
\end{proof}

\subsection{Local approximation and the proof of Theorem \ref{Lipschitz bound for 1+p Laplacian}}
\label{Subsect Local approximation}
Proposition \ref{Lipschitz Bound} states a local a priori uniform Lipschitz estimate, which is proved in Section \ref{Subsect Moser's iteration} later.
\begin{proposition}\label{Lipschitz Bound}
Let \(f\in C^{\infty}(\Omega)\).
Assume that \(u^{\epsilon}\in C^{\infty}(U)\) is a classical solution to (\ref{classical smooth equation}) in \(U\Subset\Omega\). 
Under the condition (\ref{elliptic condition})-(\ref{bound condition}), we have
\begin{equation}\label{uniform Lipschitz bound}
\sup\limits_{B_{\theta R}}\,\lvert \nabla u^{\epsilon}\rvert\le \frac{C\left(n,\,p,\,q,\,C_{1},\,C_{2}\right)}{(1-\theta)^{n/p}} \left(1+\lVert f\rVert_{L^{q}(B_{R})}^{1/(p-1)}+R^{-n/p}\left\lVert\nabla u^{\epsilon}\right\rVert_{L^{p}(B_{R})}\right)
\end{equation} for any closed ball \(B_{R}\subset U\) with \(0<R\le 1\), any \(3\le n<q\le\infty\) and \(0<\theta<1\).
Even for \(n=2\), we have for each fixed \(1<\chi<2^{\ast}=\infty\),
\begin{equation}\label{uniform Lipschitz bound n=2 Moser technique}
\sup\limits_{B_{\theta R}}\,\lvert \nabla u^{\epsilon}\rvert\le \frac{C\left(p,\,q,\,\chi,\,C_{1},\,C_{2}\right)}{(1-\theta)^{\frac{2\chi}{p(\chi-1)}}} \left(1+\lVert f\rVert_{L^{q}(B_{R})}^{1/(p-1)}+R^{-2/p}\left\lVert\nabla u^{\epsilon}\right\rVert_{L^{p}(B_{R})}\right)
\end{equation}
instead of (\ref{uniform Lipschitz bound}).
\end{proposition}
\begin{remark}\upshape
By interpolation and \cite[Chapter V, Lemma 3.1]{MR717034}, we easily obtain
\begin{equation}\label{uniform Lipschitz bound; interpolated ver}
\sup\limits_{B_{\theta R}}\,\lvert \nabla u^{\epsilon}\rvert\le C\left(n,\,p,\,q,\,s,\,\theta,\,C_{1},\,C_{2}\right)\left(1+\lVert f\rVert_{L^{q}(B_{R})}^{1/(p-1)}+R^{-n/s}\left\lVert\nabla u^{\epsilon}\right\rVert_{L^{s}(B_{R})}\right)
\end{equation}
for any closed ball \(B_{R}\subset U\) with \(0<R\le 1\), any \(2\le n<q\le\infty\), \(0<\theta<1\) and \(1\le s<p\).
\end{remark}
In Subsection \ref{Subsect De Giorgi's truncation}, we also show a local a priori uniform  Lipschitz estimate in another way. This is weaker than (\ref{uniform Lipschitz bound})-(\ref{uniform Lipschitz bound n=2 Moser technique}) though.

For the proof of Theorem \ref{Lipschitz bound for 1+p Laplacian}, we do not use a result of the strong convergence for global minimizers, given in Lemma \ref{strong convergence lemma}. Instead, we use weaker results from Lemma \ref{strong convergence lemma} and a Fatou-type estimate proved in Lemma \ref{weak+uniform bound} in Appendix.
\begin{proof}
By the H\"{o}lder inequality, it suffices to consider the case \(n<q<\infty\).
Fix \(\theta<\tau<1\) and \(B_{\tau R}\subset U\coloneqq B_{R}^{\mathrm{o}}\Subset\Omega\). Here \(B_{R}^{\mathrm{o}}\) denotes an open ball with its radius \(R\).

We first consider \(f\in C^{\infty}(\Omega)\). 
For each \(0<\epsilon\le 1\) we set 
\[u^{\epsilon}\coloneqq \argmin\left\{F_{U}^{\epsilon}(v)\mathrel{}\middle|\mathrel{} v\in u+W^{1,\,p}(U)\right\}.\]
By (\ref{growth estimate of nabla aproximated Ep}), (\ref{coercivity of approximated Ep}) and the inequalities
\[\left\langle \nabla_{z}\Psi^{\epsilon}(z_{0})\mathrel{}\middle|\mathrel{}z_{0} \right\rangle\ge 0,\quad \left\lvert \nabla_{z}\Psi^{\epsilon}(z_{0})\right\rvert\le 1\quad\textrm{for all }z_{0}\in{\mathbb R}^{n},\]
we can use results from \cite[Chapter 7.1 and 7.4]{MR2356201} to obtain \(u^{\epsilon}\in L^{\infty}_{\textrm{loc}}(U)\).
Hence by \cite[Chapter IV, Theorem 6.4]{MR0244627}, we conclude that \(u^{\epsilon}\in C^{\infty}(U)\) and \(u^{\epsilon}\) is a classical solution to (\ref{classical smooth equation}) in \(U\) (see also \cite[Chapter V, Theorem 6.1-6.3]{MR0244627}).
We note
\[u=\argmin\left\{F_{U}(v)\mathrel{}\middle|\mathrel{} v\in u+W^{1,\,p}(U)\right\}\]
by a similar argument given in the proof of Corollary \(\ref{solution is minimizer}.\) By Lemma \ref{strong convergence lemma}, we obtain \(u^{\epsilon}\rightharpoonup u\) in \(W^{1,\,p}(U)\) as \(\epsilon\to 0\). Moreover we get (\ref{convergence of minimizer}).

We define 
\[E_{U}(v)\coloneqq \beta \lVert\nabla v\rVert_{L^{1}(U)}+\int_{U}E_{p}(\nabla v)\,dx.\]
for each \(v\in W^{1,\,p}(U)\). We easily check at once that
\begin{align}\label{equivalence}
c_{1}^{1/p}\lVert \nabla v\rVert_{L^{p}(U)} &\le \left[E_{U}(v)\right]^{1/p}\le
\left[\beta\int_{U}\lvert\nabla v\rvert\,dx+c_{2}\int_{U}\lvert\nabla v\rvert^{p}\,dx \right]^{1/p}\quad (\textrm{by (\ref{elliptic p-regular}) and }\beta>0)\nonumber\\
&\le \left[(c_{2}+1)\int_{U}\lvert\nabla v\rvert^{p}\,dx+\int_{U}\beta^{p^{\prime}}\,dx \right]^{1/p}\quad (\textrm{by the Young inequality})\nonumber\\
&\le (c_{2}+1)^{1/p}\lVert \nabla v\rVert_{L^{p}(U)}+C(n,\,p)\beta^{p^{\prime}}R^{n/p} \quad (\textrm{by the Minkowski inequality})
\end{align}
for all \(v\in W^{1,\,p}(U)\).
Also we have \(E_{U}(u)=\liminf\limits_{\epsilon\to 0}E_{U}\left(u^{\epsilon}\right)\) from (\ref{convergence of minimizer}), since
\[E_{U}(u)-\int_{U}fu\,dx=F_{U}(u)=\liminf\limits_{\epsilon\to 0}F_{U}\left(u^{\epsilon}\right)=\liminf\limits_{\epsilon\to 0}E_{U}\left(u^{\epsilon}\right)-\int_{U}fu\,dx.\]
Here we have used the compact embedding \(W^{1,\,p}(U)\hookrightarrow L^{q^{\prime}}(U)\).
Combining this fact with Proposition \ref{Lipschitz Bound}, (\ref{equivalence}), and Lemma \ref{weak+uniform bound}, we obtain
\begin{align*}
\lVert\nabla u\rVert_{L^{\infty}(B_{\theta R})}&\le C(n,\,p,\,q,\,\beta,\,c_{1},\,c_{2},\,\theta,\,\tau)\left(1+\lVert f\rVert_{L^{q}(B_{\tau R})}^{1/(p-1)}+R^{-n/p}\liminf\limits_{\epsilon\to 0}\left\lVert\nabla u^{\epsilon}\right\rVert_{L^{p}(B_{\tau R})}\right)\\&\le C(n,\,p,\,q,\,\beta,\,c_{1},\,c_{2},\,\theta,\,\tau)\left(1+\lVert f\rVert_{L^{q}(U)}^{1/(p-1)}+R^{-n/p}\left[E_{U}(u)\right]^{1/p}\right)\\&\le C(n,\,p,\,q,\,\beta,\,c_{1},\,c_{2},\,\theta,\,\tau)\left(1+\lVert f\rVert_{L^{q}(B_{R})}^{1/(p-1)}+R^{-n/p}\lVert\nabla u\rVert_{L^{p}(B_{R})}\right).
\end{align*}
Hence (\ref{main local Lipschitz estimate}) holds true for \(f\in C^{\infty}(\Omega)\).

We make a density argument to complete the proof. For \(f\in L^{q}(\Omega)\,(n<q<\infty)\), fix a sequence \(\{f_{n}\}_{n=1}^{\infty}\subset C^{\infty}(\Omega)\) such that \(f_{n}\to f\,(n\to\infty)\) in \(L^{q}(\Omega)\). We define for each \(n\in{\mathbb N}\),
\[u_{n}\coloneqq \argmin\left\{\beta\int_{\Omega}\lvert \nabla v\rvert\,dx+ \int_{\Omega}E_{p}(\nabla v)\,dx-\int_{\Omega}f_{n}v\,dx\mathrel{}\middle|\mathrel{}v\in u+W_{0}^{1,\,p}(\Omega)\right\}\in u+W^{1,\,p}(\Omega).\]
By Proposition \ref{smooth approximation of subgradient vector field}, there exists a sequence \(\{Z_{n}\}_{n=1}^{\infty}\subset L^{\infty}\left(\Omega,\,{\mathbb R}^{n}\right)\) such that \((u_{n},\,Z_{n})\in W^{1,\,p}(\Omega)\times L^{\infty}\left(\Omega,\,{\mathbb R}^{n}\right)\) is a weak solution to
\[-\beta \divx (\nabla u_{n}/\lvert \nabla u_{n}\rvert)-\divx\nabla_{z} E_{p}(\nabla u_{n})\ni f_{n}\quad  \textrm{ in } \Omega.\]
From Corollary \ref{stability of solutions}, we deduce that
\[\nabla u_{n}\rightarrow \nabla u\quad \textrm{ in }\left\{\begin{array}{cc}
L^{p}(\Omega)&(p\ge 2),\\ L^{1}(\Omega)& (1<p<2),
\end{array} \right. \textrm{ as }n\to\infty\]
For \(1<p<2\), the interpolation inequality \(\lVert \nabla (u_{n}-u_{m})\rVert_{L^{p}(U)}\le \lVert \nabla (u_{n}-u_{m})\rVert_{L^{\infty}(U)}^{1-1/p}\lVert\nabla (u_{n}-u_{m})\rVert_{L^{1}(U)}^{1/p}\) and the estimate (\ref{uniform Lipschitz bound; interpolated ver}) imply that \(\nabla u_{n}\rightarrow\nabla u\,(n\to\infty)\) in \(L^{p}(U)\). 
Again by Lemma \ref{weak+uniform bound}, we obtain
\begin{align*}
\lVert\nabla u\rVert_{L^{\infty}(B_{\theta R})}&\le C(n,\,p,\,q,\,\beta,\,c_{1},\,c_{2},\,\theta,\,\tau)\left(1+\liminf_{n\to\infty}\left[\lVert f_{n}\rVert_{L^{q}(U)}^{1/(p-1)}+R^{-n/p}\lVert\nabla u_{n}\rVert_{L^{p}(U)}\right]\right)\\&=C(n,\,p,\,q,\,\beta,\,c_{1},\,c_{2},\,\theta,\,\tau)\left(1+\lVert f\rVert_{L^{q}(U)}^{1/(p-1)}+R^{-n/p}\lVert\nabla u\rVert_{L^{p}(U)}\right).
\end{align*}
This completes the proof of (\ref{main local Lipschitz estimate}). 
\end{proof}

\section{Local a priori Lipschitz bounds}
\label{Sect Uniform local Lipschitz bound}
In Section \ref{Sect Uniform local Lipschitz bound}, we prove local a priori uniform Lipschitz estimates for classical solutions to (\ref{classical smooth equation}) with \(f\in C^{\infty}(\Omega)\).
\begin{remark}\label{Nonlinear Potential Result}\upshape
For the special case (\ref{p-laplacian case}), our result of local Lipschitz estimates (\ref{uniform Lipschitz bound}) for solutions to (\ref{Regularized Crystal Eq}) can be deduced as a special case of \cite[Theorem 1.9 and 1.11]{beck2020lipschitz}.
For (\ref{p-laplacian case}) and (\ref{Example of smooth functional}), it is easily checked that
\[c_{1}\int_{1}^{\lvert z\rvert}\left(\epsilon^{2}+s^{2}\right)^{p/2-1}s\,ds\le E^{\epsilon}(z)\quad \textrm{for all }z\in{\mathbb R}^{n}\textrm{ with }\lvert z\rvert\ge 1,\]
which is described as a coercive condition of the integrand \(E^{\epsilon}\) \cite[(1.33)]{beck2020lipschitz}.
It is also noted that 
\[\frac{C_{2}\left(\epsilon^{2}+t^{2}\right)^{p/2-1}}{C_{1}\left(\epsilon^{2}+t^{2}\right)^{p/2-1}}=\frac{c_{2}(p)+\beta}{c_{1}(p)}<\infty\quad \textrm{for all }t\ge 1\]
and thus the condition \cite[(1.34)]{beck2020lipschitz} holds true. Hence as a special case of \cite[Theorem 1.11]{beck2020lipschitz}, we obtain
\begin{align*}
& c_{1}\int_{1}^{\lVert Du^{\epsilon}\rVert_{L^{\infty}(B_{R/2})}}\left(\epsilon^{2}+s^{2}\right)^{p/2-1}s\,ds\\&\quad \le C(n,\,p,\,\beta)\left[\frac{1}{{\mathcal L}^{n}(B_{R})}\int_{B_{R}}E^{\epsilon}\left(\nabla u^{\epsilon}\right)\,dx+\lVert f\rVert_{L(n,\,1)(B_{R})}^{p/(p-1)}+\lVert f\rVert_{L(n,\,1)(B_{R})}+1+\epsilon \right]
\end{align*}
for each fixed closed ball \(B_{R}\subset \Omega\).
By the Young inequality, we can easily check that
\[(\textrm{Left Hand Side})=\frac{c_{1}}{p}\left[ \left(\epsilon^{2}+\lVert \nabla u^{\epsilon}\rVert_{L^{\infty}(B_{R/2})}^{2}\right)^{p/2} -\left(\epsilon^{2}+1\right)^{p/2}\right]\ge \frac{c_{1}}{p}\left[\lVert \nabla u^{\epsilon}\rVert_{L^{\infty}(B_{R/2})}^{p}-C(p)\right]\]
and that
\[(\textrm{Right Hand Side})\le C(n,\,p,\,\beta)\left[1+R^{-n}\lVert\nabla u^{\epsilon}\rVert_{L^{p}(B_{R})}^{p}+\lVert f\rVert_{L(n,\,1)(B_{R})}^{p/(p-1)}\right].\]
Thus we obtain (\ref{Nonlinear potential result}).
For the case \(n=2\), more sophisticated estimate than (\ref{uniform Lipschitz bound n=2 Moser technique}) can similarly be concluded as a special case of \cite[Theorem 1.11]{beck2020lipschitz}. 
Their sharp results in \cite[Theorem 1.9 and 1.11]{beck2020lipschitz} cover a variety type of elliptic equations, and they can also be applied directly to (\ref{crystal model eq}). Moreover, it will work for general equations (\ref{1+p Laplacian}) or (\ref{classical smooth equation}), as long as an integrand \(E_{p}\) or \(E_{p}^{\epsilon}\) satisfies all the conditions described in \cite[Section 1.3]{beck2020lipschitz}, including the coercive condition \cite[(1.33)]{beck2020lipschitz}.

Their proofs of local Lipschitz bounds \cite[Section 3 and 4]{beck2020lipschitz}, especially nonlinear iteration arguments in \cite[Lemma 3.1, 4.7 and 4.8]{beck2020lipschitz}, contain pointwise nonlinear potential estimates. Our proofs of a priori local Lipschitz bounds in Section \ref{Sect Uniform local Lipschitz bound}, however, does not require any potential estimate. Instead, we make iteration arguments which are rather classical and elementary.
\end{remark}
\subsection{Preliminaries for proofs of local a priori uniform Lipschitz estimates}
\label{Subsect Preliminary}
Let \(u^{\epsilon}\in C^{\infty}(U)\) be a classical solution to (\ref{classical smooth equation}) in \(U\) with \(f\in C^{\infty}(\Omega)\).
For each fixed \(0<\epsilon\le 1,\,k>0,\) and \(i\in\{\,1,\,\ldots,\,n\,\}\), we define \[u_{i,\,k}\coloneqq -\left(\partial_{x_{i}}u^{\epsilon}+k\right)_{-}+\left(\partial_{x_{i}}u^{\epsilon}-k\right)_{+}\in W^{1,\,\infty}_{\textrm{loc}}(U).\]
Here \(a_{+}\coloneqq \max\{\,a,\,0\,\},\,a_{-}\coloneqq\max\{\,-a,\,0\,\}\).
We also set \[w_{k}\coloneqq k^{2}+\sum\limits_{i=1}^{n}u_{i,\,k}^{2}\in W_{\textrm{loc}}^{1,\,\infty}(U),\quad{\hat w}_{k}\coloneqq k^{2}+\lvert\nabla u^{\epsilon}\rvert^{2}\in W_{\textrm{loc}}^{1,\,\infty}(U).\]
\cite[Proposition 3.3]{MR709038} states a local a priori \(L^{\infty}\)-\(L^{p/2}\) estimate of \({\hat w}_{\epsilon}\,(0<\epsilon\le 1)\) for classical solutions to
\[-\divx\nabla_{z} E_{p}^{\epsilon}(\nabla u^{\epsilon})=0,\] 
where \(E_{p}^{\epsilon}\) satisfies (\ref{elliptic condition for relaxed p-th growth term})-(\ref{bound condition for relaxed p-th growth term}).
The proof of Proposition \ref{Lipschitz Bound} is a modification of this one. The difference is that we should avoid an analysis for degenerate points. With this in mind, we have defined the function \(w_{k}\in W_{\textrm{loc}}^{1,\,\infty}(U)\,(k\ge 1)\) such that the support of \(\nabla w_{k}\) is contained in \(\{x\in U\mid \lvert\nabla u^{\epsilon}(x)\rvert>1\}\). We also need the compatibility of \(w_{k},\,{\hat w_{k}}\,(k\ge 1)\) and \({\hat w}_{\epsilon}\,(0<\epsilon\le 1)\) on a suitable set of regular points. In Section \ref{Subsect Preliminary}, we check this compatibility.

We first get the compatibility of \(w_{k}\) and \({\hat w}_{k}\) for \(k\ge 1\), which is described in (\ref{w versus w-hat and w-check}).
\begin{lemma}\label{Wulff potential versus parabola}
Let \(\delta_{1},\,\delta_{2}>0\). For each \(t\in\mathbb{R}\), we define
\[G_{\delta_{1},\delta_{2}}(t)\coloneqq\begin{cases}
\delta_{1}^{2}+(t-\delta_{2})^{2} & (t>\delta_{2})\\
\delta_{1}^{2} & (-\delta_{2}\le t\le \delta_{2})\\
\delta_{1}^{2}+(t+\delta_{2})^{2} & (t<-\delta_{2})
\end{cases},\quad 
{\hat G}_{\delta_{1}}(t)\coloneqq \delta_{1}^{2}+t^{2}.\]
Then there exists a constant \(K=K(\delta_{2}/\delta_{1})>1\) such that
\begin{equation}\label{g vs g-hat vs g-check}
G_{\delta_{1},\,\delta_{2}}\le {\hat G}_{\delta_{1}}\le KG_{\delta_{1},\,\delta_{2}}\quad \textrm{ in } {\mathbb R}.
\end{equation}
Hence there exists a constant \(C_{0}=C_{0}(n)>1\) such that
\begin{equation}\label{w versus w-hat and w-check}
\min\left\{\,1,\,C_{0}^{\sigma}\,\right\}\cdot w_{k}^{\sigma}\le {\hat w}_{k}^{\sigma} \le\max\left\{\,1,\,C_{0}^{\sigma}\,\right\}\cdot w_{k}^{\sigma}\quad \textrm{ in }U
\end{equation}
for all \(\sigma\in\mathbb{R}\) and \(k> 0\).
\end{lemma}
\begin{proof}
\(0<G_{\delta_{1},\,\delta_{2}}\le{\hat G}_{\delta_{1}}\) in \(\mathbb{R}\) is clear by definition. It suffices to determine a constant \(K=K(\delta_{2}/\delta_{1})>1\) such that \({\hat G}_{\delta_{1}}\le KG_{\delta_{1},\,\delta_{2}}\) in \(\mathbb{R}\). We note that this is equivalent to
\[\left\{\begin{array}{cccc}
\delta_{1}^{2}+t^{2}&\le& K\delta_{1}^{2} & \text{ for all }t\in\lbrack0,\,\delta_{2}\rbrack,\\
\delta_{1}^{2}+t^{2}&\le& K\left(\delta_{1}^{2}+(t-\delta_{2})^{2}\right) & \text{ for all }t\in(\delta_{2},\,\infty),
\end{array}\right.\]
since \(G_{\delta_{1},\,\delta_{2}}\) and \({\hat G}_{\delta_{1}}\) are symmetric. Solve two inequalities for \(L>1\),
\[\left\{\begin{array}{ccccc}
0 &\le&  \inf\left\{L \delta_{1}^{2} -\delta_{1}^{2}-t^{2}\mathrel{}\middle|\mathrel{} 0\le t\le \delta_{2}\right\}&=&(L-1)\delta_{1}^{2}-\delta_{2}^{2},\\
0 &\le& \inf\left\{L\left(\delta_{1}^{2}+(t-\delta_{2})^{2}\right)-\left(\delta_{1}^{2}+t^{2}\right)\mathrel{}\middle|\mathrel{} t>\delta_{2}\right\}&=&(L-1)\delta_{1}^{2}-\displaystyle\frac{L}{L-1}\delta_{2}^{2},
\end{array}\right.\]
and then we obtain 
\[L\ge 1+\frac{(\delta_{2}/\delta_{1})^{2}}{2}\left(1+\sqrt{1+4(\delta_{2}/\delta_{1})^{-2}}\right)\eqqcolon K(\delta_{2}/\delta_{1}).\]
The constant \(K=K(\delta_{2}/\delta_{1})>1\) determined as above satisfies (\ref{g vs g-hat vs g-check}).

Now we set \(C_{0}(n)\coloneqq K\left(\sqrt{n}\right)>1\). We note
\[w_{k}=\sum\limits_{i=1}^{n}G_{k/\sqrt{n},\,k}\left(\partial_{x_{i}}u^{\epsilon}\right),\quad {\hat w}_{k}=\sum\limits_{i=1}^{n}{\hat G}_{k/\sqrt{n}}\left(\partial_{x_{i}}u^{\epsilon}\right)\]
by definition. Combining this fact with (\ref{g vs g-hat vs g-check}) implies that
\[w_{k}\le {\hat w}_{k}\le C_{0}w_{k} \quad \textrm{ in } U.\]
(\ref{w versus w-hat and w-check}) is an easy consequence of this result.
\end{proof}

It is easy to get
\[\frac{1}{2}\left(k^{2}+\lvert z\rvert^{2}\right)\le \epsilon^{2}+\lvert z\rvert^{2}\le k^{2}+\lvert z\rvert^{2}\]
for all \(0<\epsilon\le 1\) and \(z\in{\mathbb R}^{n}\) with \(\lvert z\rvert\ge k\ge 1\). This is clear by \(\epsilon^{2}\le 1\le k^{2}\) and 
\(k^{2}+\lvert z\rvert^{2}\le 2\lvert z\rvert^{2}\le 2\left(\epsilon^{2}+\lvert z\rvert^{2}\right)\).
Hence we obtain for all \(\sigma\in{\mathbb R}\) and \(k\ge 1\),
\begin{equation}\label{wk versus we}
\min\left\{\,1,\,2^{-\sigma} \,\right\}\cdot w_{k}^{\sigma}\le {\hat w}_{\epsilon}^{\sigma}\le \max\left\{\,1,\,2^{-\sigma} \,\right\}\cdot w_{k}^{\sigma}\quad \textrm{ in }\left\{x\in U\mathrel{}\middle|\mathrel{}\lvert\nabla u^{\epsilon}(x)\rvert>k \right\}.
\end{equation}
(\ref{wk versus we}) means the compatibility of \(w_{k}\,(k\ge 1)\) and \(w_{\epsilon}\,(0<\epsilon\le 1)\).

Throughout Section \ref{Sect Uniform local Lipschitz bound}, we fix 
\begin{equation}\label{determination of k}
k\coloneqq \lVert f\rVert_{L^{q}(B_{R})}^{1/(p-1)}+1\ge 1.
\end{equation}
and set a nonnegative function 
\[f_{k}\coloneqq\frac{\lvert f\rvert^{2}}{w_{k}^{p-1}}\le \left(\frac{\lvert f\rvert}{k^{p-1}}\right)^{2}.\]
By the definition of \(f_{k}\) and (\ref{determination of k}), it is obvious that 
\begin{equation}\label{Lq estimate of f_{k}}
\lVert f_{k}\rVert_{L^{q/2}(B_{R})}\le 1.
\end{equation}
We also define two constants \(0<\lambda\le\Lambda\) as
\begin{equation}\label{definitions of lambdas}
\lambda\coloneqq C_{1}\min\left\{\,1,\,C_{0}(n)^{p/2-1\,}\right\}\min\left\{\,1,\,2^{1-p/2}\,\right\},\,\Lambda\coloneqq C_{2}\max\left\{\,1,\,C_{0}(n)^{p/2-1}\,\right\}\max\left\{\,1,\,2^{1-p/2}\,\right\},
\end{equation}
which depend only on \(n,\,p,\,C_{1}\) and \(C_{2}\).
\subsection{Moser's iteration}
\label{Subsect Moser's iteration}
By Moser's iteration, we give a proof of Proposition \ref{Lipschitz Bound}.
\begin{proof}
Let \(n\ge 3\). We divide the proof of (\ref{uniform Lipschitz bound}) into \(2\) Steps.

\underline{Step \(1\)}. 
The aim of Step \(1\) is to prove the following Caccioppoli-type inequality.
\begin{equation}\label{Caccioppoli-type inequality for simple case}
\int_{B_{R}}\lvert\nabla (\eta v_{\alpha})\rvert^{2}\,dx \le C\left(n,\,p,\,q,\,\lambda,\,\Lambda\right)(1+\alpha)^{\beta}\int_{B_{R}}v_{\alpha}^{2}\left(\lvert\nabla\eta\rvert^{2}+\eta^{2}\right)\,dx
\end{equation}
for any \(\alpha\ge 0\) and \(\eta\in C_{c}^{1}(B_{R})\), where \(v_{\alpha}\coloneqq w_{k}^{(\alpha+p)/4}\in W^{1,\infty}_{\textrm{loc}}(\Omega)\) and \(\beta=\beta(n,\,q)\ge 2\) is a constant to be chosen later.

We prove (\ref{Caccioppoli-type inequality for simple case}) by a standard absorbing argument. 
For each fixed \(i\in\{\,1,\,\ldots,\,n\,\}\), differentiate (\ref{classical smooth equation}) with respect to \(x_{i}\). Then using integration by parts, we have
\begin{equation}\label{Euler-Lagrange 1}
\int_{B_{R}}\left\langle\nabla_{z}^{2}E^{\epsilon}\left(\nabla u^{\epsilon}\right)\nabla\partial_{x_{i}}u^{\epsilon}\mathrel{}\middle|\mathrel{}\nabla\phi\right\rangle\,dx+\int_{B_{R}}f\partial_{x_{i}}\phi\,dx=0
\end{equation}
for all \(\phi\in W_{0}^{1,\,p}(B_{R})\). We test \(\phi\coloneqq u_{i,\,k}w_{k}^{\alpha/2}\eta^{2}\in W_{0}^{1,\,p}(B_{R})\) in (\ref{Euler-Lagrange 1}). We note that \(\phi\) is supported in the superlevelset \(\left\{x\in U\mid \partial_{x_{i}}u^{\epsilon}(x)>k\right\}\), and hence we can replace \(\nabla\partial_{x_{i}}u^{\epsilon}\) by \(\nabla u_{i,\,k}\). Summing over \(i\in\{\,1,\,\ldots,\,n\,\}\), we obtain
\begin{align}\label{tested equation 1}
&\underbrace{\int_{B_{R}}w_{k}^{\alpha/2}\eta^{2}\sum\limits_{i=1}^{d}\left\langle\nabla_{z}^{2}E^{\epsilon}(\nabla u^{\epsilon})\nabla u_{i,\,k}\mathrel{}\middle|\mathrel{}\nabla u_{i,\,k} \right\rangle\,dx+\frac{1}{2}\int_{B_{R}}\eta^{2}\left\langle\nabla_{z}^{2}E^{\epsilon}(\nabla u^{\epsilon})\nabla w_{k}\mathrel{}\middle|\mathrel{}\nabla w_{k}^{\alpha/2} \right\rangle\,dx}_{\eqqcolon I_{1}}\nonumber\\&+\underbrace{\int_{B_{R}}w_{k}^{\alpha/2}\eta \left\langle\nabla_{z}^{2}E^{\epsilon}(\nabla u^{\epsilon})\nabla w_{k}\mathrel{}\middle|\mathrel{}\nabla\eta\right\rangle\,dx}_{\eqqcolon I_{2}}+\underbrace{\int_{B_{R}}f\sum\limits_{i=1}^{d}\partial_{x_{i}}\left(u_{i,\,k}w_{k}^{\alpha/2}\eta^{2}\right)\,dx}_{\eqqcolon I_{3}}=0.
\end{align}
We set an integral
\[J_{\alpha}\coloneqq \int_{B_{R}}\eta^{2}w_{k}^{(\alpha+p)/2-1}\sum\limits_{i=1}^{n}\lvert\nabla u_{i,\,k}\rvert^{2}\,dx+\frac{\alpha}{4}\int_{B_{R}}\eta^{2}w_{k}^{(\alpha+p)/2-2}\lvert\nabla w_{k}\rvert^{2}\,dx.\]
It is easy to obtain \(I_{1}\ge\lambda J_{\alpha}\) from (\ref{elliptic condition}), (\ref{w versus w-hat and w-check})-(\ref{wk versus we}) and (\ref{definitions of lambdas}). We estimate \(\lvert I_{3}\rvert\) as following;
\begin{align}
\lvert I_{3}\rvert&
\le C(n)\int_{B_{R}}\lvert f\rvert\left(\eta^{2}w_{k}^{\alpha/2}\left(\sum_{i=1}^{n}\lvert\nabla u_{i,\,k}\rvert^{2}\right)^{1/2}+\alpha\eta^{2}w_{k}^{(\alpha-1)/2}\lvert\nabla w_{k}\rvert+w_{k}^{(\alpha+1)/2}\lvert\eta\rvert\lvert\nabla\eta\rvert\right)\,dx\nonumber\\ &\quad(\textrm{by the Cauchy-Schwarz inequality})\nonumber\\&
\le \frac{\lambda}{2}J_{\alpha}+\frac{C(n)}{\lambda}(1+\alpha)\int_{B_{R}}\eta^{2}\lvert f\rvert^{2}w_{k}^{(\alpha-p)/2+1}\,dx\nonumber\\ &\quad
+C(n)\left(\frac{1}{1+\alpha}\int_{B_{R}}w_{k}^{(\alpha+p)/2}\lvert\nabla\eta\rvert^{2}\,dx+(1+\alpha)\int_{B_{R}}\eta^{2}\lvert f\rvert^{2}w_{k}^{(\alpha-p)/2+1}\,dx\right)\nonumber\\ &\quad (\textrm{by the Young inequality})\nonumber\\
&\le\frac{\lambda}{2}J_{\alpha}+C(n,\,\lambda)(1+\alpha)\left(\int_{B_{R}}v_{\alpha}^{2}\lvert\nabla\eta\rvert^{2}\,dx+\int_{B_{R}}f_{k} \eta^{2}v_{\alpha}^{2}\,dx\right).\nonumber\\& \quad (\textrm{by the definitions of }f_{k},\,v_{\alpha})\nonumber
\end{align}
From (\ref{tested equation 1}) we deduce that
\[\frac{\lambda}{2}J_{\alpha}\le\lvert I_{2}\rvert+C(n,\,\lambda)\left(\int_{B_{R}}v_{\alpha}^{2}\lvert\nabla\eta\rvert^{2}\,dx+(1+\alpha)\int_{B_{R}}f_{k}\eta^{2}v_{\alpha}^{2}\,dx\right).\]
The Cauchy-Schwarz inequality implies 
\[\lvert\nabla w_{k}\rvert^{2}=4\sum\limits_{j=1}^{n}\left(\sum\limits_{i=1}^{n}u_{i,\,k}\partial_{x_{j}}u_{i,\,k}\right)^{2}\le 4\sum\limits_{j=1}^{n}\left(\sum\limits_{i=1}^{n}u_{i,\,k}^{2}\right)\left(\sum\limits_{i=1}^{n}\lvert\partial_{x_{j}}u_{i,\,k}\rvert^{2}\right)\le 4w_{k}\sum\limits_{i=1}^{n}\lvert\nabla u_{i,\,k}\rvert^{2},\]
which leads to
\[J_{\alpha}\ge\frac{1+\alpha}{4}\int_{B_{R}}\eta^{2}w_{k}^{(\alpha+p)/2-2}\lvert\nabla w_{k}\rvert^{2}\,dx.\]
Combining this inequality with (\ref{bound condition}), (\ref{w versus w-hat and w-check})-(\ref{wk versus we}), (\ref{definitions of lambdas}) and the Young inequality, we have
\begin{align*}
\lvert I_{2}\rvert&\le \Lambda\int_{B_{R}}w_{k}^{(\alpha+p)/2-1}\lvert\eta\rvert \lvert\nabla w_{k}\rvert\lvert\nabla\eta\rvert\,dx\\&\le \frac{\lambda(1+\alpha)}{16}\int_{B_{R}}\eta^{2}w_{k}^{(\alpha+p)/2-2}\lvert\nabla w_{k}\rvert^{2}\,dx+\frac{4\Lambda^{2}}{\lambda(1+\alpha)}\int_{B_{R}}w_{k}^{(\alpha+p)/2}\lvert\nabla\eta\rvert^{2}\,dx\\&\le \frac{\lambda}{4}J_{\alpha}+\frac{4\Lambda^{2}}{\lambda}\int_{B_{R}}v_{\alpha}^{2}\lvert\nabla\eta\rvert^{2}\,dx.
\end{align*}
Therefore we obtain
\[\int_{B_{R}}\eta^{2}w_{k}^{(\alpha+p)/2-2}\lvert\nabla w_{k}\rvert^{2}\,dx\le C(n,\,\lambda,\,\Lambda)\left(\int_{B_{R}}v_{\alpha}^{2}\lvert\nabla\eta\rvert^{2}\,dx+\int_{B_{R}}f_{k}\eta^{2}v_{\alpha}^{2}\,dx\right).\]
We note that by direct calculation
\[\lvert\nabla v_{\alpha}\rvert^{2}=\frac{(\alpha+p)^{2}}{16}w_{k}^{(\alpha+p)/2-2}\lvert\nabla w_{k}\rvert^{2}\le C(p)(1+\alpha)^{2}w_{k}^{(\alpha+p)/2-2}\lvert\nabla w_{k}\rvert^{2}.\]
From this it follows that
\[\int_{B_{R}}\eta^{2}\lvert\nabla v_{\alpha}\rvert^{2}\,dx \le C(n,\,p,\,\lambda,\,\Lambda)(1+\alpha)^{2}\left(\int_{B_{R}}v_{\alpha}^{2}\lvert\nabla\eta\rvert^{2}\,dx+\int_{B_{R}}f_{k}\eta^{2}v_{\alpha}^{2}\,dx\right),\]
and hence
\begin{equation}\label{pre Caccioppoli-type estimate for simple case}
\int_{B_{R}}\lvert\nabla (\eta v_{\alpha})\rvert^{2}\,dx \le C(n,\,p,\,\lambda,\,\Lambda)(1+\alpha)^{2}\left(\int_{B_{R}}v_{\alpha}^{2}\lvert\nabla\eta\rvert^{2}\,dx+\int_{B_{R}}f_{k}\eta^{2}v_{\alpha}^{2}\,dx\right).
\end{equation}
For \(q=\infty\), it is easy to check from (\ref{Lq estimate of f_{k}}) and (\ref{pre Caccioppoli-type estimate for simple case}), that (\ref{Caccioppoli-type inequality for simple case}) is valid with \(\beta=2\).
For \(3\le n<q<\infty\), by the H\"{o}lder inequality and (\ref{Lq estimate of f_{k}}), we have
\[\int_{B_{R}}f_{k}\eta^{2}v_{\alpha}^{2}\,dx\le \lVert f_{k}\rVert_{L^{q/2}(B_{R})}\left[\int_{B_{R}}\left(\eta^{2}v_{\alpha}^{2}\right)^{\frac{2q}{q-2}}\,dx\right]^{1-2/q}\le \lVert \eta v_{\alpha} \rVert_{L^{\frac{q}{q-2}}(B_{R})}^{2}.\]
Interpolation with \(L^{2}\subset L^{\frac{2q}{q-2}}\subset L^{2^{\ast}}\) (Note that \(2<\frac{2q}{q-2}<\frac{2n}{n-2}=2^{\ast}\) since \(3\le n<q<\infty\)) and the Sobolev embedding \(W_{0}^{1,2}(B_{R})\hookrightarrow L^{2^{\ast}}(B_{R})\) imply that
\[\lVert \eta v_{\alpha} \rVert_{L^{\frac{2q}{q-2}}(B_{R})}^{2}\le \delta \int_{B_{R}}\lvert\nabla (\eta v_{\alpha})\rvert^{2}\,dx +C(n,\,q)\delta^{-\frac{n}{q-n}}\int_{B_{R}}\eta^{2}v_{\alpha}^{2}\,dx\]
for any small number \(\delta>0\). Take \(\delta=c(1+\alpha)^{-2}\) with \(c=c(n,p,q,\lambda,\Lambda)>0\) sufficiently small, then from (\ref{pre Caccioppoli-type estimate for simple case}) we obtain (\ref{Caccioppoli-type inequality for simple case}) with \(\beta\coloneqq 2q/(q-n)\ge 2\). 

\underline{Step \(2\)}. 
From (\ref{Caccioppoli-type inequality for simple case}) we prove a local \(L^{\infty}\)-\(L^{p/2}\) estimate of \(w_{k}\) by Moser's iteration.

Set \(\chi(n)\coloneqq 2^{\ast}/2=n/(n-2)\in(1,\,\infty)\). We claim a reversed H\"{o}lder inequality
\begin{equation}\label{reversed Hoelder inequality simple case}
\lVert w_{k}\rVert_{L^{\frac{(\alpha+p)}{2}\chi}(B_{\rho})}\le \left[C(n,\,p,\,q,\,\lambda,\,\Lambda)\frac{(1+\alpha)^{\beta}}{(r-\rho)^{2}}\right]^{2/(\alpha+p)}\lVert w_{k}\rVert_{L^{\frac{(\alpha+p)}{2}}(B_{r})}
\end{equation}
for all \(0<\rho<r\le R\) and \(\alpha\ge 0\). For any fixed \(0<\rho<r\le R\), we take a cutoff function \(\eta\in C_{c}^{1}(B_{r})\) such that 
\begin{equation}\label{cut-off function}
0\le \eta \le 1 \textrm{ in }B_{r},\,
\eta=1 \textrm{ in }B_{\rho}, \textrm{ and }
\lvert\nabla\eta\rvert\le\frac{2}{r-\rho} \textrm{ in }B_{r}.
\end{equation}
Then we obtain
\begin{align*}
\left(\int_{B_{\rho}}w_{k}^{(\alpha+p)\chi/2}\,dx\right)^{1/\chi}&
\le \left(\int_{B_{r}}(\eta v_{\alpha})^{2\chi}\,dx\right)^{1/\chi}\le C(n)\int_{B_{r}}\lvert\nabla (\eta v_{\alpha})\rvert^{2}\,dx\\ &
\quad \left(\textrm{by the Sobolev embedding } W_{0}^{1,\,2}(B_{r})\hookrightarrow L^{2\chi}(B_{r})=L^{2^{\ast}}(B_{r})\right)\\&
\le C(n,\,p,\,q,\,\lambda,\,\Lambda)(1+\alpha)^{\beta}\int_{B_{r}}v_{\alpha}^{2}\left(\lvert\nabla\eta\rvert^{2}+\eta^{2}\right)\,dx \quad (\textrm{by } (\ref{Caccioppoli-type inequality for simple case}))\\ &
\le C(n,\,p,\,q,\,\lambda,\,\Lambda)(1+\alpha)^{\beta} \left(\frac{1}{(r-\rho)^{2}}+1\right)\int_{B_{r}}v_{\alpha}^{2}\,dx \quad (\textrm{by } (\ref{cut-off function}))\\ &
\le C(n,\,p,\,q,\,\lambda,\,\Lambda)\frac{(1+\alpha)^{\beta}}{(r-\rho)^{2}}\int_{B_{r}}w_{k}^{(\alpha+p)/2}\,dx \quad (\textrm{note }0<r-\rho<1),
\end{align*}
which implies (\ref{reversed Hoelder inequality simple case}).

For each \(N\in{\mathbb N}\cup\{0\}\), we define
\begin{equation}\label{sequences for the Moser iteration}
\alpha_{N}\coloneqq p\left(\chi^{N}-1\right),\,\gamma_{N}\coloneqq \frac{p}{2}\chi^{N},\textrm{ and } r_{N}\coloneqq \left[\theta +2^{-N}(1-\theta) \right]R.
\end{equation}
We note \(\alpha_{N}+p=p\chi^{N}\) for each \(N\in{\mathbb N}\cup\{0\}\).
Applying (\ref{reversed Hoelder inequality simple case}) with \((\alpha,\,\rho,\,r)=(\alpha_{N},\,r_{N+1},\,r_{N})\), we have for all \(N\in{\mathbb N}\cup\{0\}\)
\begin{align*}
\lVert w_{k}\rVert_{L^{\gamma_{N+1}}\left(B_{r_{N+1}}\right)}&\le \left[C(n,\,p,\,q,\,\lambda,\,\Lambda)\frac{\left(p\chi^{N}-p+1\right)^{\beta}}{\left[2^{-(N+1)}(1-\theta)R\right]^{2}}\right]^{2/\left(p\chi^{N}\right)}\lVert w_{k}\rVert_{L^{\gamma_{N}}\left(B_{r_{N}}\right)}\nonumber\\ &\le \left[\frac{C_{\dagger}^{N}}{[(1-\theta)R]^{4/p}}\right]^{\chi^{-N}}\lVert w_{k}\rVert_{L^{\gamma_{N}}\left(B_{r_{N}}\right)}
\end{align*}
for some \(C_{\dagger}=C_{\dagger}(n,\,p,\,q,\,\lambda,\,\Lambda)>0\). By iteration we can check that for each \(N\in{\mathbb N}\),
\begin{align*}
\lVert w_{k}\rVert_{L^{\gamma_{N}}(B_{\theta R})} 
&\le \lVert w_{k}\rVert_{L^{\gamma_{N}}\left(B_{r_{N}}\right)}\\
&\le [(1-\theta)R]^{-\frac{4}{p}\sum\limits_{j=0}^{\infty}\chi^{-j}}\underbrace{\left[C_{\dagger}\right]^{\sum\limits_{j=0}^{\infty}j\chi^{-j}}}_{\eqqcolon C_{\dagger\dagger}<\infty}\lVert w_{k}\rVert_{L^{\gamma_{0}}\left(B_{r_{0}}\right)}=C_{\dagger\dagger}(n,\,p,\,q,\,\lambda,\,\Lambda)\frac{\lVert w_{k}\rVert_{L^{p/2}(B_{R})}}{[(1-\theta)R]^{2n/p}}.
\end{align*}
Letting \(N\to\infty\), we obtain
\[\lVert w_{k}\rVert_{L^{\infty}(B_{\theta R})}\le C_{\dagger\dagger}(n,\,p,\,q,\,\lambda,\,\Lambda)\frac{\lVert w_{k}\rVert_{L^{p/2}(B_{R})}}{[(1-\theta)R]^{2n/p}}.\]
Combining this result with (\ref{w versus w-hat and w-check}) and the Minkowski inequality, we have
\begin{align*}
\sup\limits_{B_{\theta R}}\,\left\lvert\nabla u^{\epsilon}\right\rvert& \le C_{0}(n)^{1/2}\lVert w_{k}\rVert_{L^{\infty}(B_{\theta R})}^{1/2}\\&
\le C(n,\,p,\,q,\,\lambda,\,\Lambda)\frac{\lVert w_{k}\rVert_{L^{p/2}(B_{R})}^{1/2}}{[(1-\theta)R]^{n/p}}=\frac{C(n,\,p,\,q,\,\lambda,\,\Lambda)}{[(1-\theta)R]^{n/p}} \left(\int_{B_{R}}\left(k^{2}+\left\lvert \nabla u^{\epsilon}\right\rvert^{2} \right)^{p/2}\,dx \right)^{1/p}\\&
\le \frac{C(n,\,p,\,q,\,\lambda,\,\Lambda)}{(1-\theta)^{n/p}}\left(k+R^{-n/p}\left\lVert \nabla u^{\epsilon}\right\rVert_{L^{p}(B_{R})}\right).
\end{align*}
Recall (\ref{determination of k}), and it completes the proof of (\ref{uniform Lipschitz bound}).

We give a proof of (\ref{uniform Lipschitz bound n=2 Moser technique}) by making modifications of arguments given in Step \(1\) and \(2\).
We note that (\ref{pre Caccioppoli-type estimate for simple case}) is valid even for \(n=2\) by the same computations, but that \(W_{0}^{1,\,2}(B_{R})\hookrightarrow L^{\infty}(B_{R})\) does not hold. 

In Step 1, for \(2<q<\infty\), fix an arbitary constant \(\frac{2q}{q-2}<\kappa<\infty\), make an interpolation with \(L^{2}\subset L^{\frac{2q}{q-2}}\subset L^{\kappa}\) and apply the Sobolev embedding \(W_{0}^{1,\,2}(B_{R})\hookrightarrow L^{\kappa}(B_{R})\). Then from (\ref{pre Caccioppoli-type estimate for simple case}), we obtain (\ref{Caccioppoli-type inequality for simple case}) for some \(\beta=\beta(q,\,\kappa)\ge 2+\frac{4}{q-2}\). For \(q=\infty\), from (\ref{Lq estimate of f_{k}}) and (\ref{pre Caccioppoli-type estimate for simple case}), we can check that (\ref{Caccioppoli-type inequality for simple case}) is valid with \(\beta=2\), similarly as \(n\ge 3\).

In Step 2, fix an arbitary \(\chi\in(1,\,\infty)\) and apply the Sobolev embedding \(W_{0}^{1,\,2}(B_{r})\hookrightarrow L^{2\chi}(B_{r})\). Then we obtain an alternative reversed H\"{o}lder inequality,
\[\lVert w_{k}\rVert_{L^{\frac{(\alpha+p)}{2}\chi}(B_{\rho})}\le \left[C(p,\,q,\,\chi,\,\lambda,\,\Lambda)\frac{(1+\alpha)^{\beta}}{(r-\rho)^{2}}r^{2/\chi}\right]^{2/(\alpha+p)}\lVert w_{k}\rVert_{L^{\frac{(\alpha+p)}{2}}(B_{r})}\]
instead of (\ref{reversed Hoelder inequality simple case}). Set \(\gamma_{N},\, r_{N}\) as in (\ref{sequences for the Moser iteration}), then we get
\begin{align*}
\lVert w_{k}\rVert_{L^{\gamma_{N+1}}\left(B_{r_{N+1}}\right)}&
\le \left[C(p,\,q,\,\chi,\,\lambda,\,\Lambda)\frac{\left(p\chi^{N}-p+1\right)^{\beta}}{\left[2^{-(N+1)}(1-\theta)R\right]^{2}}R^{2/\chi}\right]^{2/\left(p\chi^{N}\right)}\lVert w_{k}\rVert_{L^{\gamma_{N}}\left(B_{r_{N}}\right)}\\& 
\le \left[\frac{\left(C_{\dagger}(p,\,q,\,\chi,\,\lambda,\,\Lambda)\right)^{N}}{\left[(1-\theta)R\right]^{4(1-1/\chi)/p}} (1-\theta)^{-4/(p\chi)} \right]^{\chi^{-N}}\lVert w_{k}\rVert_{L^{\gamma_{N}}\left(B_{r_{N}}\right)}.
\end{align*}
By iteration we can check that for each \(N\in{\mathbb{N}}\),
\begin{align*}
\lVert w_{k}\rVert_{L^{\gamma_{N}}(B_{\theta R})} &\le\lVert w_{k}\rVert_{L^{\gamma_{N}}\left(B_{r_{N}}\right)}
\\ &\le [(1-\theta)R]^{-\frac{4}{p}\left(1-\frac{1}{\chi}\right)\sum\limits_{j=0}^{\infty}\chi^{-j}}\underbrace{\left[C_{\dagger}\right]^{\sum\limits_{j=0}^{\infty}j\chi^{-j}}}_{\eqqcolon C_{\dagger\dagger}<\infty}\lVert w_{k}\rVert_{L^{\gamma_{0}}\left(B_{r_{0}}\right)}\cdot (1-\theta)^{-\frac{4}{p\chi}\sum\limits_{j=0}^{\infty}\chi^{-j}}
\\&=C_{\dagger\dagger}(p,\,q,\,\chi,\,\lambda,\,\Lambda)\frac{\lVert w_{k}\rVert_{L^{p/2}(B_{R})}}{[(1-\theta)R]^{4/p}}\cdot (1-\theta)^{-\frac{4}{p(\chi-1)}},
\end{align*}
from which we conclude (\ref{uniform Lipschitz bound n=2 Moser technique}), similarly for \(n\ge 3\). 
\end{proof}

\begin{remark}\label{case beta=0}\upshape
Consider \(\beta=0\). Then (\ref{1+p Laplacian}) becomes
\begin{equation}\label{p-Laplacian}
\divx\nabla_{z} E_{p}(\nabla u)=f\quad \text{\rm in } \Omega.
\end{equation}
If \(u\in W^{1,\,p}(\Omega)\) is a weak solution to (\ref{p-Laplacian}), then  by making some modifications we conclude that
\begin{equation}\label{Local Lipschitz estimate for p-laplacian}
\lVert\nabla u\rVert_{L^{\infty}(B_{\theta R})}\le C(n,\,p,\,q,\,c_{1},\,c_{2},\,\theta)\left(\lVert f\rVert_{L^{q}(B_{R})}^{1/(p-1)}+R^{-n/p}\lVert\nabla u\rVert_{L^{p}(B_{R})}\right)
\end{equation}
for any fixed closed ball \(B_{R}\subset \Omega\) with its radius \(0<R\le 1\), any \(2\le n<q\le\infty\) and \(0<\theta<1\),
instead of (\ref{main local Lipschitz estimate}). Estimates of the type (\ref{Local Lipschitz estimate for p-laplacian}) can be seen in \cite[Theorem 1.15]{beck2020lipschitz} for solutions to uniformly elliptic systems. Compared with the proof of \cite[Theorem 1.15]{beck2020lipschitz}, our proof of (\ref{Local Lipschitz estimate for p-laplacian}) is rather direct.
\end{remark}
We give a sketch of the proof of (\ref{Local Lipschitz estimate for p-laplacian}). 
\begin{proof}
We claim the following a priori estimate.
\begin{equation}\label{Perturbation a priori estimate}
\sup\limits_{B_{\theta R}}\,\lvert \nabla u^{\epsilon}\rvert\le C\left(n,\,p,\,q,\,c_{1},\,c_{2},\,\theta\right)\left(\delta+\lVert f\rVert_{L^{q}(B_{R})}^{1/(p-1)}+R^{-n/p}\left\lVert\nabla u^{\epsilon}\right\rVert_{L^{p}(B_{R})}\right)
\end{equation}
for any \(0<\epsilon<\delta<1\), any closed ball \(B_{R}\subset U\) with \(0<R\le 1\), any \(3\le n<q\le\infty\) and \(0<\theta<1\). Here \(u^{\epsilon}\in C^{\infty}(U)\) is a classical solution to
\[-\divx \nabla_{z}E_{p}^{\epsilon}\left(\nabla u^{\epsilon}\right)=f\in C^{\infty}(\Omega)\quad \textrm{ in }U\Subset \Omega.\]
For each fixed \(0<\delta<1\), we set 
\[k\coloneqq \delta+\lVert f\rVert_{L^{q}(B_{R})}^{1/(p-1)}\ge \delta,\quad f_{k}\coloneqq\frac{\lvert f\rvert^{2}}{w_{k}^{p-1}}\le \left(\frac{\lvert f\rvert}{k^{p-1}}\right)^{2}.\]
We note that (\ref{wk versus we}) is valid for all \(\sigma\in{\mathbb R}\) and \(0<\epsilon\le\delta\le k\).
Using (\ref{elliptic condition for relaxed p-th growth term})-(\ref{bound condition for relaxed p-th growth term}), (\ref{w versus w-hat and w-check})-(\ref{wk versus we}) and (\ref{definitions of lambdas}), we deduce (\ref{Perturbation a priori estimate}), as in the proof of Proposition \ref{Lipschitz Bound}.
Recalling all the proofs, we can easily check that Corollary \ref{solution is minimizer}-\ref{stability of solutions} and Lemma \ref{lowersemicontinuity of F}, \ref{strong convergence lemma} are valid even for \(\beta=0\). As in the proof of Theorem \ref{Lipschitz bound for 1+p Laplacian}, we conclude from (\ref{Perturbation a priori estimate}) that
\[\lVert\nabla u\rVert_{L^{\infty}(B_{\theta R})}\le C(n,\,p,\,q,\,c_{1},\,c_{2},\,\theta)\left(\delta+\lVert f\rVert_{L^{q}(B_{R})}^{1/(p-1)}+R^{-n/p}\lVert\nabla u\rVert_{L^{p}(B_{R})}\right)\]
for all \(f\in L^{q}(\Omega)\,(n<q\le\infty)\) and for any fixed \(0<\delta<1\), from which we obtain (\ref{Local Lipschitz estimate for p-laplacian}).
\end{proof}
\begin{remark}\label{How to modify Kruegel's proof}\upshape
The significant difference between the proof of Proposition \ref{Lipschitz Bound} and that of \cite[Lemma 4.9]{krugel2013variational} is that different test functions are chosen. For simplicity, let \(f=\const\eqqcolon a\). In Kr\"{u}gel's essential proof, we test \(\phi\coloneqq \left(u_{i,\,1}\right)_{+}\left[w_{1}^{+}\right]^{\alpha/2}\eta^{2}\in W_{0}^{1,\,p}(B_{R})\) or \(\phi\coloneqq -\left(u_{i,\,1}\right)_{-}\left[w_{1}^{-}\right]^{\alpha/2}\eta^{2}\in W_{0}^{1,\,p}(B_{R})\) in (\ref{Euler-Lagrange 1}). Here \(\alpha\ge 0\), and the functions \(w_{1}^{+},\,w_{1}^{-}\) are defined as
\[w_{1}^{+}\coloneqq 1+\sum\limits_{i=1}^{n}\left(u_{i,\,1}\right)_{+}^{2}\in W_{\textrm{loc}}^{1,\,\infty}(U)\quad \textrm{and} \quad w_{1}^{-}\coloneqq 1+\sum\limits_{i=1}^{n}\left(u_{i,\,1}\right)_{-}^{2}\in W_{\textrm{loc}}^{1,\,\infty}(U).\]
From this, we make a similar absorbing argument for alternative integrals
\[J_{\alpha}^{\pm}\coloneqq \int_{B_{R}}\eta^{2}{\hat w}_{1}^{p/2-1}\left[ w_{1}^{\pm}\right]^{\alpha/2}\sum\limits_{i=1}^{n}\left\lvert\nabla (u_{i,\,1})_{\pm}\right\rvert^{2}\,dx+\frac{\alpha}{4}\int_{B_{R}}\eta^{2}{\hat w}_{1}^{p/2-1}\left[w_{1}^{\pm}\right]^{\alpha/2-1}\left\lvert\nabla w_{1}^{\pm}\right\rvert^{2}\,dx.\]
Then, as in Step 1 of the proof of Proposition \ref{Lipschitz Bound}, we obtain 
\begin{equation}\label{Correct Kruegel estimate}
\int_{B_{R}}\eta^{2}{\hat w}_{1}^{p/2-1}\left(w_{1}^{\pm}\right)^{\alpha/2} \left\lvert\nabla w_{1}^{\pm}\right\rvert^{2}\,dx\le C(a,\,n,\,p,\,C_{1},\,C_{2})\int_{B_{R}}{\hat w}_{1}^{p/2}\left[w_{1}^{\pm}\right]^{\alpha/2}\left(\eta^{2}+\lvert\nabla\eta\rvert^{2}\right)\,dx
\end{equation}
from (\ref{elliptic condition})-(\ref{bound condition}) and (\ref{wk versus we}). Here we note that (\ref{w versus w-hat and w-check}) is not used. Kr\"{u}gel claims that
\begin{equation}\label{Kruegel estimate}
\int_{B_{R}}\eta^{2}\left(w_{1}^{\pm}\right)^{(p+\alpha)/2-2}\left\lvert\nabla w_{1}^{\pm} \right\rvert^{2}\,dx\le C(a,\,n,\,p,\,C_{1},\,C_{2})\int_{B_{R}}\left[w_{1}^{\pm}\right]^{(\alpha+p)/2}\left(\eta^{2}+\lvert\nabla\eta\rvert^{2}\right)\,dx
\end{equation}
by (\ref{Correct Kruegel estimate}). From (\ref{Kruegel estimate}) we conclude that
\[\int_{B_{R}}\lvert\nabla (\eta v_{\alpha})\rvert^{2}\,dx \le C(a,\,n,\,p,\,C_{1},\,C_{2})(1+\alpha)^{2}\int_{B_{R}}v_{\alpha}^{2}\left(\lvert\nabla\eta\rvert^{2}+\eta^{2}\right)\,dx,\quad\textrm{where}\quad v_{\alpha}\coloneqq w_{1}^{(\alpha+p)/4}\]
instead of (\ref{pre Caccioppoli-type estimate for simple case}), since we easily obtain
\[\int_{B_{R}}\eta^{2}w_{1}^{(\alpha+p)/2-2}\left\lvert\nabla w_{1}\right\rvert^{2}\,dx \le C(a,\,n,\,p,\,C_{1},\,C_{2})\int_{B_{R}} w_{1}^{(\alpha+p)/2}\left(\eta^{2}+\lvert\nabla\eta\rvert^{2}\right)\,dx.\]
The rest of Kr\"{u}gel's proof is very similar to that of Proposition \ref{Lipschitz Bound}. Hence it suffices to prove (\ref{Kruegel estimate}) from (\ref{Correct Kruegel estimate}). The problem is, however, that neither \(w_{k}^{+}\) nor \(w_{k}^{-}\) is compatible with \({\hat w}_{k}\). That is, though it is clear that \(w_{k}^{\pm}\le {\hat w}_{k}\) in \(U\), there does not exist \(C=C(n)>1\) such that \({\hat w}_{k}\le Cw_{k}^{\pm}\) in \(U\). This makes it difficult to obtain (\ref{Kruegel estimate}) from (\ref{Correct Kruegel estimate}), if \(p/2-1<0\), i.e. \(1<p<2\). We overcome this problem by taking other suitable test functions carefully.
\end{remark}

\subsection{De Giorgi's truncation}
\label{Subsect De Giorgi's truncation}
By De Giorgi's truncation, it is possible to obtain another local a priori uniform Lipschitz estimate, which is much rougher than the results in Proposition \ref{Lipschitz Bound}. More general and sophisticated estimate via De Giorgi's truncation can be seen in the recent work by Beck and Mingione \cite[Section 4.2]{beck2020lipschitz}, while our approach is rather classical and elementary.
\begin{proposition}\label{De Giorgi's truncation proof}
Let \(f\in C^{\infty}(\Omega)\).
Assume that \(u^{\epsilon}\in C^{\infty}(U)\) is a classical solution to (\ref{classical smooth equation}) in \(U\Subset\Omega\).
Under the condition (\ref{elliptic condition}) and (\ref{bound condition}), we have
\begin{equation}\label{uniform Lipschitz bound, weaker}
\sup\limits_{B_{\theta R}}\,\lvert \nabla u^{\epsilon}\rvert\le \frac{C\left(n,\,p,\,q,\,C_{1},\,C_{2}\right)}{[(1-\theta)R]^{g}} \left( R^{n/p}\left[1+\lVert f\rVert_{L^{q}(B_{R})}^{1/(p-1)}\right]+\left\lVert\nabla u^{\epsilon}\right\rVert_{L^{p}(B_{R})}\right)
\end{equation} for any closed ball \(B_{R}\subset \Omega\) with \(0<R\le 1\), any \(2\le n<q\le\infty\) and \(0<\theta<1\). Here \(g=g(n,\,p,\,q)\ge n/p\) is a constant.
\end{proposition}
\begin{proof}
We divide the proof into 3 Steps.

\underline{Step 1}. 
We set \(V_{l}\coloneqq \left(w_{k}^{p/2}-l\right)_{+}\in W_{\textrm{loc}}^{1,\,\infty}(\Omega)\) and \(A(l,\,r)\coloneqq \left\{x\in B_{r}\mathrel{} \middle| \mathrel{} w_{k}(x)>l^{2/p}\right\}\) for \(l\ge 0\) and \(0<r\le R\).
The aim of Step 1 is to prove that there exists a constant \(\gamma=\gamma(n,\,q)\in (0,\,2/n\rbrack\) such that
\begin{equation}\label{De Giorgi class estimate}
\int_{A(l,\,r)}(\eta V_{l})^{2}\,dx\le C(n,\,p,\,q,\,\lambda,\,\Lambda)\left( {\mathcal L}^{n}\left(A(l,\,r)\right)^{\gamma} \int_{A(l,\,r)}V_{l}^{2}\lvert\nabla \eta\rvert^{2}\,dx +l^{2}{\mathcal L}^{n}\left(A(l,\,r)\right)^{1+\gamma} \right)
\end{equation}
for all \(0<r<R\), \(\eta\in C_{c}^{1}(B_{r},\,\lbrack 0,\,1\rbrack)\) and \(l\ge l_{0}\coloneqq C_{\ast}\lVert V_{0}\rVert_{L^{2}(B_{R})}\).
Here \({\mathcal L}^{n}\) denotes \(n\)-dimensional Lebesgue measure, and \(C_{\ast}=C_{\ast}(n,\,p,\,q,\,\lambda,\,\Lambda)>0\) is a constant which is chosen later.

We test \(\phi\coloneqq u_{i,\,k}V_{l}\eta^{2}\in W_{0}^{1,\,p}(B_{R})\) in (\ref{Euler-Lagrange 1}). We note that all integrals range over the superlevelset \(\left\{x\in U\mid \partial_{x_{i}}u^{\epsilon}(x)>k\right\}\cap A(l,\,r)\) and therefore we may replace \(\nabla\partial_{x_{i}}u^{\epsilon},\,\nabla w_{k}\) by \(\nabla u_{i,\,k},\,\nabla \left(w_{k}-l^{2/p}\right)_{+}\) respectively. By summing over \(i\in\{\,1,\,\ldots,\,n\,\}\), we obtain 
\begin{align}\label{tested equation 2}
&\underbrace{\int_{A(l,\,r)}\eta^{2}V_{l}\sum\limits_{i=1}^{d}\left\langle\nabla_{z}^{2}E^{\epsilon}(\nabla u^{\epsilon})\nabla u_{i,\,k}\mathrel{}\middle|\mathrel{}\nabla u_{i,\,k} \right\rangle\,dx+\frac{1}{2}\int_{A(l,\,r)}\eta^{2}\left\langle\nabla_{z}^{2}E^{\epsilon}(\nabla u^{\epsilon})\nabla \left(w_{k}-l^{2/p}\right)_{+}\mathrel{}\middle|\mathrel{}\nabla V_{l} \right\rangle\,dx}_{\eqqcolon I_{1}}\nonumber\\
&+\underbrace{\int_{A(l,\,r)}\eta V_{l} \left\langle\nabla_{z}^{2}E^{\epsilon}(\nabla u^{\epsilon})\nabla \left(w_{k}-l^{2/p}\right)_{+}\mathrel{}\middle|\mathrel{}\nabla\eta\right\rangle\,dx}_{\eqqcolon I_{2}}+\underbrace{\int_{A(l,\,r)}f\sum\limits_{i=1}^{d}\partial_{x_{i}}\left(\eta^{2}u_{i,\,k}V_{l}\right)\,dx}_{\eqqcolon I_{3}}=0.
\end{align}
We set an integral
\begin{equation}\label{definition of local energy functional J}
J\coloneqq \int_{A(l,\,r)}\eta^{2}w_{k}^{p/2-1}V_{l}\sum\limits_{i=1}^{n}\lvert\nabla u_{i,\,k}\rvert^{2}\,dx+\frac{p}{4}\int_{A(l,\,r)}\eta^{2}w_{k}^{p-2}\left\lvert\nabla \left(w_{k}-l^{2/p}\right)_{+}\right\rvert^{2}\,dx.
\end{equation}
\(I_{1}\ge\lambda J\) is easily obtained from (\ref{elliptic condition}), (\ref{w versus w-hat and w-check})-(\ref{wk versus we}) and (\ref{definitions of lambdas}).
We note that 
\[w_{k}^{p/2}=V_{l}+l \textrm{ on } A(l,\,r), \textrm{ and hence } w_{k}^{p}\le 2\left(V_{l}^{2}+l^{2}\right)\textrm{ on } A(l,\,r).\]
With this in mind, we obtain
\begin{align}
\lvert I_{3}\rvert&
\le C(n)\int_{A(l,\,r)}\lvert f\rvert\left(\eta^{2}V_{l}\left(\sum_{i=1}^{n}\lvert\nabla u_{i,\,k}\rvert^{2}\right)^{1/2}+p\eta^{2}w_{k}^{(p-1)/2}\left\lvert\nabla \left(w_{k}-l^{2/p}\right)_{+}\right\rvert+V_{l}w_{k}^{1/2}\eta\lvert\nabla\eta\rvert\right)\,dx\nonumber\\ &\quad(\textrm{by the Cauchy-Schwarz inequality})\nonumber\\&
\le \frac{\lambda}{2}J+\frac{C(n)}{\lambda}\left( \int_{A(l,\,r)}\eta^{2}w_{k}^{1-p/2}\lvert f\rvert^{2}V_{l}\,dx+p\int_{A(l,\,r)}\eta^{2}w_{k}\lvert f\rvert^{2}\,dx \right)\nonumber\\ &\quad
+\frac{C(n)}{2}\left(\int_{A(l,\,r)}V_{l}^{2}\lvert\nabla\eta\rvert^{2}\,dx+\int_{A(l,\,r)}\eta^{2}w_{k}\lvert f\rvert^{2}\,dx\right)\nonumber\\ &\quad (\textrm{by the Young inequality})\nonumber\\
&\le\frac{\lambda}{2}J+C(n,\,p,\,\lambda)\left(\int_{A(l,\,r)}V_{l}^{2}\lvert\nabla\eta\rvert^{2}\,dx+\int_{A(l,\,r)}\eta^{2}f_{k}\left(w_{k}^{p}+V_{l}w_{k}^{p/2}\right)\,dx\right)\nonumber\\& \quad (\textrm{by the definition of }f_{k})\nonumber\\ 
&\le \frac{\lambda}{2}J+C(n,\,p,\,\lambda) \left(\int_{A(l,\,r)}V_{l}^{2}\lvert\nabla\eta\rvert^{2}\,dx+\int_{A(l,\,r)}\eta^{2}f_{k}\left(V_{l}^{2}+lV_{l}+l^{2}\right)\,dx\right).
\nonumber
\end{align}
From this and (\ref{tested equation 2}), we deduce that 
\[\frac{\lambda}{2}J\le \lvert I_{2}\rvert +C(n,\,p,\,\lambda)\left(\int_{A(l,\,r)}V_{l}^{2}\lvert\nabla\eta\rvert^{2}\,dx+\int_{A(l,\,r)}\eta^{2}f_{k}\left(V_{l}^{2}+lV_{l}+l^{2}\right)\,dx\right).\]
By dropping the first term in (\ref{definition of local energy functional J}), we have
\[J\ge \frac{p}{4}\int_{A(l,\,r)}\eta^{2}w_{k}^{p-2}\left\lvert\nabla \left(w_{k}-l^{2/p}\right)_{+}\right\rvert^{2}\,dx =\frac{1}{p}\int_{A(l,\,r)}\eta^{2}\lvert \nabla V_{l}\rvert^{2}\,dx.\]
By (\ref{bound condition}), (\ref{w versus w-hat and w-check})-(\ref{wk versus we}), (\ref{definitions of lambdas}) and the Young inequality, we obtain
\begin{align*}
\lvert I_{2}\rvert&\le \Lambda\int_{A(l,\,r)}\eta\lvert\nabla\eta\rvert V_{l}w_{k}^{p/2-1}\left\lvert\nabla \left(w_{k}-l^{2/p}\right)_{+}\right\rvert\,dx
=\frac{2\Lambda}{p}\int_{A(l,\,r)}\eta V_{l}\lvert\nabla\eta\rvert \lvert\nabla V_{l}\rvert\,dx\\
&\le \frac{\lambda}{4p}\int_{A(l,\,r)}\eta^{2}\lvert\nabla V_{l}\rvert^{2}\,dx +\frac{4\Lambda^{2}}{\lambda p}\int_{A(l,\,r)}V_{l}^{2}\lvert\nabla\eta\rvert^{2}\,dx.
\end{align*}
From these it follows that
\[\int_{A(l,\,r)}\eta^{2}\lvert\nabla V_{l}\rvert^{2}\,dx\le C(n,\,p,\,\lambda,\,\Lambda) \left(\int_{A(l,\,r)}V_{l}^{2}\lvert\nabla\eta\rvert^{2}\,dx+\int_{A(l,\,r)}\eta^{2}f_{k}\left(V_{l}^{2}+lV_{l}+l^{2}\right)\,dx\right),\]
and hence
\begin{equation}\label{pre estimate De giorgi 1}
\int_{A(l,\,r)}\lvert\nabla (\eta V_{l})\rvert^{2}\,dx\le C(n,\,p,\,\lambda,\,\Lambda) \left(\int_{A(l,\,r)}V_{l}^{2}\lvert\nabla\eta\rvert^{2}\,dx+\int_{A(l,\,r)}\eta^{2}f_{k}\left(V_{l}^{2}+lV_{l}+l^{2}\right)\,dx\right).
\end{equation}
From (\ref{pre estimate De giorgi 1}) we verify that (\ref{De Giorgi class estimate}) is valid. We first consider \(n\ge 3\). By \(0\le \eta\le 1\), (\ref{Lq estimate of f_{k}}), the H\"{o}lder inequality and the Sobolev embedding \(W_{0}^{1,\,2}(B_{r})\hookrightarrow L^{2^{\ast}}(B_{r})\), we obtain 
\[\int_{A(l,\,r)}\eta^{2}f_{k}\,dx\le \lVert f_{k}\rVert_{L^{q/2}(B_{r})} {\mathcal L}^{n}(A(l,\,r))^{1-2/q}\le {\mathcal L}^{n}(A(l,\,r))^{1-2/q},\]
\begin{align*}
\int_{A(l,\,r)}\eta^{2}f_{k}V_{l}^{2}\,dx&\le \lVert f_{k}\rVert_{L^{q/2}(B_{r})}\left(\int_{B_{r}} (\eta V_{l})^{2^{\ast}}\right)^{2/2^{\ast}}{\mathcal L}^{n}(A(l,\,r))^{2/n-2/q}\\
&\le C(n){\mathcal L}^{n}(A(l,\,r))^{2/n-2/q} \int_{A(l,\,r)}\lvert\nabla(\eta V_{l})\rvert^{2}\,dx,
\end{align*}
\begin{align*}
l\int_{A(l,\,r)}\eta^{2}f_{k}V_{l}\,dx&\le l\lVert f_{k}\rVert_{L^{q/2}(B_{r})} C(n)\left(\int_{B_{r}} \lvert\nabla(\eta V_{l})\rvert^{2}\right)^{1/2}{\mathcal L}^{n}(A(l,\,r))^{1-1/2^{\ast}-2/q}\\
&\le \delta \int_{A(l,\,r)}\lvert\nabla(\eta V_{l})\rvert^{2}\,dx+\frac{C(n)}{\delta} l^{2} {\mathcal L}^{n}(A(l,\,r))^{1+2/n-4/q} 
\end{align*}
for any \(\delta>0\). Take \(\delta=\delta(n,\,p,\,\lambda,\,\Lambda)>0\) sufficiently small, and assume that 
\begin{equation}\label{condition for l0}
{\mathcal L}^{n}(A(l,\,r))\le c(n,\,p,\,q,\,\lambda,\,\Lambda)
\end{equation}
for some sufficiently small constant \(0<c<1\). Then from (\ref{pre estimate De giorgi 1}), we have
\[\int_{A(l,\,r)}\lvert\nabla (\eta V_{l})\rvert^{2}\,dx\le C(n,\,p,\,\lambda,\,\Lambda) \left(\int_{A(l,\,r)}V_{l}^{2}\lvert\nabla\eta\rvert^{2}\,dx+l^{2} {\mathcal L}^{n}(A(l,\,r))^{1-2/q}\right).\]
We can choose \(C_{\ast}\) such that the assumption (\ref{condition for l0}) holds true for all \(l\ge l_{0}=C_{\ast}\lVert V_{0}\rVert_{L^{2}(B_{R})}\), since by the H\"{o}lder inequality we obtain
\[{\mathcal L}^{n}(A(l,\,r))\le \frac{1}{l}\int_{A(l,\,r)}w_{k}^{p/2}\,dx\le \frac{{\mathcal L}^{n}(A(l,\,r))^{1/2}}{l}\lVert V_{0}\rVert_{L^{2}(B_{R})} \textrm{ and hence } {\mathcal L}^{n}(A(l,\,r))\le \left(\frac{\lVert V_{0}\rVert_{L^{2}(B_{R})}}{l}\right)^{2}.\]
Again by the H\"{o}lder inequality and the Sobolev embedding \(W_{0}^{1,\,2}(B_{r})\hookrightarrow L^{2^{\ast}}(B_{r})\), we get
\begin{align*}
\int_{A(l,\,r)} (\eta V_{l})^{2}\,dx&\le C(n) {\mathcal L}^{n}(A(l,\,r))^{2/n} \int_{A(l,\,r)}\lvert\nabla(\eta V_{l})\rvert^{2}\,dx\\
&\le C(n,\,p,\,\lambda,\,\Lambda) \left({\mathcal L}^{n}(A(l,\,r))^{2/n}\int_{A(l,\,r)}V_{l}^{2}\lvert\nabla\eta\rvert^{2}\,dx+l^{2} {\mathcal L}^{n}(A(l,\,r))^{1+2/n-2/q}\right).
\end{align*}
for all \(l\ge l_{0}\). We note (\ref{condition for l0}), and hence conclude that \(\gamma(n,\,q)\coloneqq 2/n-2/q>0\) satisfies (\ref{De Giorgi class estimate}) for \(n\ge 3\). For \(n=2\), fix \(2<\kappa<\infty\) and use the Sobolev embedding \(W_{0}^{1,\,2}(B_{r})\hookrightarrow L^{\kappa}(B_{r})\). Then, by a similar argument we realize that (\ref{De Giorgi class estimate}) is valid for some \(\gamma=\gamma(q,\,\kappa)\in (0,\,1-2/q)\).

\underline{Step 2}. The aim of Step 2 is to prove that
\[\int_{A(L_{0}+l_{0},\,\theta R)}V_{L_{0}+l_{0}}^{2}\,dx=0\]
for \(L_{0}=C_{\star}\lVert V_{0}\rVert_{L^{2}(B_{R})}\). Here \(C_{\star}=C_{\star}(n,\,p,\,q,\,\lambda,\,\Lambda,\,\theta,\,R)>0\) is a constant which is chosen later.
For any fixed \(0<\rho<r\le R\), take a cutoff function \(\eta\in C_{c}^{1}(B_{r},\,\lbrack 0,\,1\rbrack)\) as in (\ref{cut-off function}).
We note that for any \(L>l\ge l_{0}\),
\[{\mathcal L}^{n}(A(L,\,r))={\mathcal L}^{n}\left(\left\{x\in B(r)\mathrel{}\middle| \mathrel{} w_{k}^{p/2}-l>L-l\right\}\right)\le \frac{1}{(L-l)^{2}}\int_{A(l,\,r)}V_{l}^{2}\,dx, \textrm{ and }\]
\[\int_{A(L,\,r)}V_{L}^{2}\,dx\le \int_{A(l,\,r)}V_{l}^{2}\,dx\quad (\textrm{since }A(L,\,r)\subset A(l,\,r),\,V_{L}\le V_{l}\textrm{ in } \Omega).\]
Hence by (\ref{De Giorgi class estimate}), we obtain 
\begin{align}\label{Recursive inequality}
\int_{A(L,\,\rho)}V_{L}^{2}\,dx&
\le \int_{A(L,\,\rho)}(\eta V_{L})^{2}\,dx\nonumber\\ &
\le C(n,\,p,\,q,\,\lambda,\,\Lambda)\left( \int_{A(L,\,r)}V_{L}^{2}\lvert\nabla \eta\rvert^{2}\,dx +L^{2}{\mathcal L}^{n}\left(A(L,\,r)\right) \right){\mathcal L}^{n}\left(A(L,\,r)\right)^{\gamma} \nonumber\\&
\le C(n,\,p,\,q,\,\lambda,\,\Lambda)\left[\frac{1}{(r-\rho)^{2}}+\frac{L^{2}}{(L-l)^{2}}\right]\frac{1}{(L-l)^{2\gamma}}\left(\int_{A(l,\,r)}V_{l}^{2}\,dx\right)^{1+\gamma}
\end{align}
for any \(L> l\ge l_{0}\) and \(0<\rho<r\le R\).
Now we use an iteration argument. For each \(N\in{\mathbb N}\cup \{0\}\), set 
\[l_{N}\coloneqq l_{0}+L_{0}\left(1-2^{-N}\right),\, r_{N}\coloneqq \left\{ \theta +2^{-N}(1-\theta) \right\}R,\textrm{ and } a_{N}\coloneqq \lVert V_{l_{N}}\rVert_{L^{2}\left(B_{r_{N}}\right)}.\]
By (\ref{Recursive inequality}), we get 
\[a_{N+1}\le C(n,\,p,\,q,\,\lambda,\,\Lambda)\left[\frac{2^{N+1}}{(1-\theta)R}+2^{N+1}\right]\frac{2^{\gamma(N+1)}}{L_{0}^{\gamma}}a_{N}\le \frac{C_{\dagger}(n,\,p,\,q,\,\lambda,\,\Lambda)}{(1-\theta)R}L_{0}^{-\gamma}2^{(1+\gamma)N}a_{N}^{1+\gamma}\]
for any \(n\in{\mathbb N}\cup\{0\}\).
Set 
\[L_{0}\coloneqq \underbrace{\left[\frac{C_{\dagger}}{(1-\theta)R}\right]^{1/\gamma}2^{\frac{1+\gamma}{\gamma^{2}}}}_{\eqqcolon C_{\star}}\lVert V_{0}\rVert_{L^{2}(B_{R})}\ge C_{\star}\lVert V_{l_{0}}\rVert_{L^{2}(B_{R})}.\]
Then we obtain 
\[a_{0}=\lVert V_{l_{0}}\rVert_{L^{2}(B_{R})}\le \lVert V_{0}\rVert_{L^{2}(B_{R})}=\left[\frac{C_{\dagger}}{(1-\theta)R}L_{0}^{-\gamma}\right]^{-1/\gamma}\left(2^{1+\gamma}\right)^{-1/\gamma^{2}}\]
and hence \(a_{N}\to 0\) as \(N\to\infty\), by \cite[Chapter 2, Lemma 4.7]{MR0244627}. From this we have
\[0\le \int_{A(L_{0}+l_{0},\,\theta R)}V_{L_{0}+l_{0}}^{2}\,dx\le \liminf\limits_{N\to\infty}\int_{A(l_{N},\,r_{N})}V_{l_{N}}^{2}\,dx =\liminf\limits_{N\to\infty} a_{N}^{2}=0,\]
which implies that 
\begin{equation}\label{L infty-L2 estimate}
\lVert V_{0}\rVert_{L^{\infty}(B_{\theta R})}=\left\lVert w_{k}\right\rVert_{L^{\infty}(B_{\theta R})}^{p/2}\le (C_{\ast}+C_{\star})\lVert V_{0}\rVert_{L^{2}(B_{R})}.
\end{equation}

\underline{Step 3}. 
Set \(g\coloneqq 2/(p\gamma)\ge n/p\).
We make an interpolation argument to prove
\begin{equation}\label{L infty-L1 estimate}
\lVert V_{0}\rVert_{L^{\infty}(B_{\theta R})}\le C(n,\,p,\,q,\,\lambda,\,\Lambda)\frac{\lVert V_{0}\rVert_{L^{1}(B_{R})}}{[(1-\theta)R]^{pg}}.
\end{equation}
We note that 
\(\lVert V_{0}\rVert_{L^{2}(B_{R})}\le \lVert V_{0}\rVert_{L^{1}(B_{R})}^{1/2}\lVert V_{0}\rVert_{L^{\infty}(B_{R})}^{1/2}\). By (\ref{L infty-L2 estimate}) and the Young inequality, we obtain
\begin{align*}
\lVert V_{0}\rVert_{L^{\infty}(B_{\theta R})}&\le \left(C_{\ast}+ \left[\frac{C_{\dagger}}{(1-\theta)R}\right]^{1/\gamma}2^{\frac{1+\gamma}{\gamma^{2}}}\right)\lVert V_{0}\rVert_{L^{2}(B_{R})}\le \frac{C(n,\,p,\,q,\,\lambda,\,\Lambda)}{[(1-\theta)R]^{pg/2}}\lVert V_{0}\rVert_{L^{1}(B_{R})}^{1/2}\lVert V_{0}\rVert_{L^{\infty}(B_{R})}^{1/2}\\ &\le \frac{1}{2}\lVert V_{0}\rVert_{L^{\infty}(B_{R})} +C(n,\,p,\,q,\,\lambda,\,\Lambda)\frac{\lVert V_{0}\rVert_{L^{1}(B_{R})}}{[(1-\theta)R]^{pg}}.
\end{align*} Hence (\ref{L infty-L1 estimate}) follows from \cite[Chapter V, Lemma 3.1]{MR717034}. By (\ref{w versus w-hat and w-check}), (\ref{L infty-L1 estimate}) and the Minkowski inequality, we obtain
\begin{align*}
\sup\limits_{B_{\theta R}}\lvert\nabla u^{\epsilon}\rvert &
\le C_{0}(n)^{1/2}\lVert V_{0}\rVert_{L^{\infty}(B_{\theta R})}^{1/p}\\&
\le \frac{C(n\,,p\,,q\,,\lambda\,,\Lambda)}{[(1-\theta)R]^{g}}\lVert V_{0}\rVert_{L^{1}(B_{R})}^{1/p}
=\frac{C(n,\,p,\,q,\,\lambda,\,\Lambda)}{[(1-\theta)R]^{g}} \left(\int_{B_{R}}\left(k^{2}+\lvert \nabla u^{\epsilon}\rvert^{2} \right)^{p/2}\,dx \right)^{1/p}\\&
\le \frac{C(n,\,p,\,q,\,\lambda,\,\Lambda)}{[(1-\theta)R]^{g}}\left( R^{n/p}k+\left\lVert\nabla u^{\epsilon}\right\rVert_{L^{p}(B_{R})}\right).
\end{align*}
Recall (\ref{determination of k}), and it completes the proof of (\ref{uniform Lipschitz bound, weaker}). 
\end{proof}
\begin{remark}\upshape
If \(n\ge 3\) and \(q=\infty\), we may take \(g=n/p\) and therefore (\ref{uniform Lipschitz bound}) is obtained.
\end{remark}

\begin{acknowledgements}\upshape
It is a pleasure to acknowledge the many helpful suggestions of Professor Yoshikazu Giga during the preparation of the paper. The author is grateful to Professor Giuseppe Mingione, who kindly pointed
out his results very related to the present paper.
\end{acknowledgements}

\appendix
\section{Elementary proofs of three lemmas}
In Appendix, we give precise proofs of three lemmas for completeness. Most of the proofs are elementary in the sense that we just use standard tools of calculus, measure theory, convex analysis, functional analysis and real analysis.
\subsection{Vector inequalities}
Vector inequalities (\ref{vectror inequalities on coercivity})-(\ref{growth estimate of approximated Ep}) are used throughout Section \ref{Sect weak solution}, \ref{Sect Approximation schemes}.
\begin{lemma}\label{Vector inequalities}
Let \(E_{p},\,E_{p}^{\epsilon}\) satisfy (\ref{elliptic p-regular})-(\ref{convergence condition of nabla p-th growth term}) and (\ref{values of Ep and nabla Ep at 0}). Then we obtain inequalities (\ref{vectror inequalities on coercivity})-(\ref{coercivity of approximated Ep}).
\end{lemma}
For the special case (\ref{p-laplacian case}) and \(\epsilon=1\), proofs of (\ref{vectror inequalities on coercivity}) are given \cite[Section 12]{MR2242021}, and \cite[Lemma 13.3 and 30.1]{oden1986qualitative}. Here we give a generalized proof of inequalities (\ref{vectror inequalities on coercivity})-(\ref{coercivity of approximated Ep}) via smooth approximation.
\begin{proof}
By (\ref{convergence condition of nabla p-th growth term}), the proof of (\ref{vectror inequalities on coercivity}) is completed by showing that 
\begin{equation}\label{Coercive vector ineq; approximated}
\left\langle \nabla_{z}E_{p}^{\epsilon}(z_{2})-\nabla_{z}E_{p}^{\epsilon}(z_{1})\mathrel{}\middle|\mathrel{} z_{2}-z_{1}\right\rangle\ge \left\{\begin{array}{cc}
c_{1}\cdot C(p)\lvert z_{1}-z_{2}\rvert^{p} & (2\le p<\infty),\\
c_{1}\lvert z_{1}-z_{2}\rvert^{2}\left(\epsilon^{2}+\lvert z_{1}\rvert^{2}+\lvert z_{2}\rvert^{2}\right)^{p/2-1} & (1<p<2),
\end{array} \right.\quad \textrm{ for all }z_{1},\,z_{2}\in{\mathbb R}^{n},\,0<\epsilon\le 1.
\end{equation}
For each fixed \(0<\epsilon<1\) and \(z_{1},\,z_{2}\in{\mathbb R}^{n}\), we have
\begin{align*}
\left\langle \nabla_{z}E_{p}^{\epsilon}(z_{2})-\nabla_{z}E_{p}^{\epsilon}(z_{1})\mathrel{}\middle|\mathrel{} z_{2}-z_{1}\right\rangle&=\int_{0}^{1}\left\langle \nabla_{z}^{2}E_{p}^{\epsilon}(tz_{2}+(1-t)z_{1}) (z_{2}-z_{1})\mathrel{}\middle|\mathrel{} z_{2}-z_{1}\right\rangle\,dt\\&\ge c_{1}\lvert z_{1}-z_{2}\rvert^{2}\int_{0}^{1}\left(\epsilon^{2}+\lvert tz_{2}+(1-t)z_{1}\rvert^{2} \right)^{p/2-1}\,dt.
\end{align*}
Here we have used (\ref{elliptic condition for relaxed p-th growth term}).
(\ref{Coercive vector ineq; approximated}) for \(1<p<2\) is easily obtained by a simple inequality \[\left(\epsilon^{2}+\lvert tz_{2}+(1-t)z_{1}\rvert^{2} \right)^{p/2-1}\ge \left(\epsilon^{2}+\lvert z_{1}\rvert^{2}+\lvert z_{2}\rvert^{2} \right)^{p/2-1}\quad\textrm{ for }0\le t\le 1.\]
Even for \(2\le p<\infty\), we get (\ref{Coercive vector ineq; approximated}) as following.
\[\lvert z_{1}-z_{2}\rvert^{2}\int_{0}^{1}\left(\epsilon^{2}+\lvert tz_{2}+(1-t)z_{1}\rvert^{2} \right)^{p/2-1}\,dt\ge \int_{0}^{1}\lvert tz_{2}+(1-t)z_{1}\rvert^{p-2}\,dt\ge C(p)\lvert z_{1}-z_{2}\rvert^{p}.\]
For the last inequality, see the proof of \cite[Chapter I, Lemma 4.4]{MR1230384}.

We note that (\ref{bound condition for relaxed p-th growth term}) implies that for each fixed \(z_{0}\in{\mathbb R}^{n}\), eigenvalues of the positive real symmetric matrix \(\nabla_{z}^{2}E_{p}^{\epsilon}(z_{0})\) are all no greater than \(c_{2}\left(\epsilon^{2}+\lvert z_{0}\rvert^{2}\right)^{p/2-1}\). Hence for all \(z_{1},\,z_{2}\in{\mathbb R}^{n}\), \(0<\epsilon \le 1\),
\[\left\lvert \nabla_{z}E_{p}^{\epsilon}(z_{1})-\nabla_{z}E_{p}^{\epsilon}(z_{2})\right\rvert=\left\lvert\int_{0}^{1}\nabla_{z}^{2}E_{p}^{\epsilon}(tz_{1}+(1-t)z_{2})\cdot t(z_{1}-z_{2})\,dt\right\rvert\le c_{2}\lvert z_{1}-z_{2}\rvert\int_{0}^{1}\left(\epsilon^{2}+\lvert z_{2}+t(z_{1}-z_{2})\rvert^{2}\right)^{p/2-1}t\,dt.\]
For \(p\ge 2\), we easily obtain (\ref{uniform continuity of nabla approximated Ep}) by using a simple inequality
\[\left(\epsilon^{2}+\lvert tz_{2}+(1-t)z_{1}\rvert^{2} \right)^{p/2-1}\le C(p)\left(\epsilon^{p-2}+\lvert z_{1}\rvert^{p-2}+\lvert z_{2}\rvert^{p-2} \right)\quad\textrm{ for all }0\le t\le 1.\]
For \(1<p<2\), we get (\ref{uniform continuity of nabla approximated Ep}) as following,
\[\lvert z_{1}-z_{2}\rvert\int_{0}^{1}\left(\epsilon^{2}+\lvert z_{2}+t(z_{1}-z_{2})\rvert^{2}\right)^{p/2-1}t\,dt\le \lvert z_{1}-z_{2}\rvert\int_{0}^{1}\lvert z_{2}+t(z_{1}-z_{2})\rvert^{p-2}\,dt\le C(p)\lvert z_{1}-z_{2}\rvert^{p-1}.\]
For the last inequality, see the proof of \cite[Chapter I, Lemma 4.4]{MR1230384}.
Let \((z_{1},\,z_{2})=(z_{0},\,0)\). Then we obtain (\ref{growth estimate of nabla aproximated Ep}) for \(1<p<2\) as following.
\[\left\lvert\nabla_{z} E_{p}^{\epsilon}(z_{0})-\nabla_{z} E_{p}^{\epsilon}(0)\right\rvert\le c_{2}\cdot C(p)\left(\epsilon^{p-2}\lvert z_{0}\rvert+\lvert z_{0}\rvert^{p-1}\right)\le c_{2}\cdot C(p)\left(\epsilon^{p-1}+\lvert z_{0}\rvert^{p-1}\right).\]
Here we have used the Young inequality. The proof of (\ref{growth estimate of nabla aproximated Ep}) for \(2\le p<\infty\) is similar. Letting \(\epsilon\to 0\), we conclude (\ref{growth estimate of nabla Ep}) from (\ref{values of Ep and nabla Ep at 0}) and (\ref{growth estimate of nabla aproximated Ep}).

(\ref{growth estimate of approximated Ep}) is easily deduced from (\ref{growth estimate of nabla aproximated Ep}). For \(2\le p<\infty\), we calculate
\begin{align*}
\left\lvert E_{p}^{\epsilon}(z_{0})-E_{p}^{\epsilon}(0)\right\rvert&=\left\lvert \int_{0}^{1}\left\langle\nabla_{z}E_{p}^{\epsilon}(tz_{0})\mathrel{}\middle|\mathrel{}tz_{0}\right\rangle\,dt \right\rvert\\&\le \left\lvert \nabla_{z}E_{p}^{\epsilon}(0)\right\rvert\lvert z_{0}\rvert\int_{0}^{1}t\,dt+\lvert z_{0}\rvert\int_{0}^{1}\left\lvert \nabla_{z}E_{p}^{\epsilon}(tz_{0})-\nabla_{z}E_{p}^{\epsilon}(0)\right\rvert t\,dt\\&\le \frac{1}{2}\left\lvert \nabla_{z}E_{p}^{\epsilon}(0)\right\rvert\lvert z_{0}\rvert +c_{2}\cdot C(p)\int_{0}^{1}\left(\epsilon^{p-1}t\lvert z_{0}\rvert+t^{p}\lvert z_{0}\rvert^{p} \right)\,dt\\&\le C(c_{2},\,p)\left(\epsilon^{p-1}\lvert z_{0}\rvert+\left\lvert \nabla_{z}E_{p}^{\epsilon}(0)\right\rvert\lvert z_{0}\rvert+\lvert z_{0}\rvert^{p} \right)\quad \textrm{ for all }z\in{\mathbb R}^{n}
\end{align*}
The proof of (\ref{growth estimate of approximated Ep}) for \(1<p<2\) is similar.

For the proof of (\ref{coercivity of approximated Ep}), we note that \(\partial E_{p}^{\epsilon}(z_{0})=\left\{\nabla_{z}E_{p}^{\epsilon}(z_{0})\right\}\) for all \(z_{0}\in{\mathbb R}^{n}\), since \(E_{p}^{\epsilon}\in C^{\infty}\) is convex. Hence by the subgradient inequality, we obtain
\[E_{p}^{\epsilon}(z_{0})-E_{p}^{\epsilon}(0)-\left\langle\nabla_{z}E_{p}^{\epsilon}(0) \mathrel{}\middle|\mathrel{} z_{0}\right\rangle\ge \left\langle \nabla_{z}E_{p}^{\epsilon}(z_{0})-\nabla_{z}E_{p}^{\epsilon}(0)\mathrel{}\middle|\mathrel{} z_{0}\right\rangle\] for all \(z_{0}\in{\mathbb R}^{n}\).
We can also easily check at once that
\[\lvert z_{0}\rvert^{2}\left(\epsilon^{2}+\lvert z_{0}\rvert^{2}\right)^{p/2-1}=\left(\epsilon^{2}+\lvert z_{0}\rvert^{2}\right)^{p/2}-\epsilon^{2}\left(\epsilon^{2}+\lvert z_{0}\rvert^{2}\right)^{p/2-1}\ge \left(\epsilon^{2}+\lvert z_{0}\rvert^{2}\right)^{p/2}-\epsilon^{p}\]
for all \(0<\epsilon<1,\,z_{0}\in{\mathbb R}^{n}\) and \(1<p<2\).
Combining these inequalities with (\ref{Coercive vector ineq; approximated}), we conclude (\ref{coercivity of approximated Ep}). 
\end{proof}
\subsection{A justification for convergence of minimizers}
Lemma \ref{strong convergence lemma} is used in Section \ref{Sect Approximation schemes} to justify that a sequence of local or global minimizers \(\left\{u^{\epsilon}\right\}_{0<\epsilon\le 1}\) converges to a minimizer \(u\).
\begin{lemma}\label{strong convergence lemma}
Let \(E_{p},\,\left\{E_{p}^{\epsilon}\right\}_{0<\epsilon\le 1}\) satisfy (\ref{elliptic p-regular})-(\ref{convergence condition of nabla p-th growth term}). For bounded domain \(V\subset {\mathbb R}^{n}\) with Lipschitz boundary, assume that \(u\in W^{1,\,p}(V)\) satisfies
\begin{equation}\label{minimizer of Fv}
u=\argmin\left\{F_{V}(v)\mathrel{}\middle|\mathrel{}v\in u+W_{0}^{1,\,p}(V)\right\}.
\end{equation}
For each \(0<\epsilon\le 1\), we define
\begin{equation}\label{minimizer of approximated Fv}
u^{\epsilon}\coloneqq \argmin\left\{F_{V}^{\epsilon}(v)\mathrel{}\middle|\mathrel{}v\in u+W_{0}^{1,\,p}(V)\right\}\in u+W_{0}^{1,\,p}(V).
\end{equation}
Then \(u^{\epsilon}\rightharpoonup u\,(\epsilon\to 0)\) in \(W^{1,\,p}(V)\) and
\begin{equation}\label{convergence of minimizer}
\lim\limits_{\epsilon\to 0}F_{V}\left(u^{\epsilon}\right)=\lim\limits_{\epsilon\to 0}F_{V}^{\epsilon}\left(u^{\epsilon}\right)=\lim\limits_{\epsilon\to 0}F_{V}^{\epsilon}(u)=F_{V}(u).
\end{equation}
Moreover, up to a subsequence we obtain \(u^{\epsilon}\to u\) in \(W^{1,\,p}(V)\).
\end{lemma}
In \cite[Theorem 3.3]{krugel2013variational}, Kr\"{u}gel, inspired by the proof of \cite[Theorem 6.1]{MR1451535}, discussed weak or strong convergence of minimizers, for the special case where \(E_{p}\) and \(E^{\epsilon}_{p}\) are sphere symmetric and \(f=\const\) For the reader's convenience, we give a proof of Lemma \ref{strong convergence lemma} by generalizing Kr\"{u}gel's idea.
\begin{proof}
We first note that
\begin{equation}\label{condition for approximated energy functional}
F_{V}(v)\le F_{V}^{\epsilon}(v)\to F_{V}(v)\quad \textrm{ as }\epsilon\to 0
\end{equation}
for each fixed \(v\in W^{1,\,p}(V)\). 
(\ref{condition for approximated energy functional}) is clear by (\ref{convergence condition of p-th growth term}), (\ref{uniformly bound estimate approximated Ep and nabla Ep at 0}), (\ref{growth estimate of approximated Ep}) and Lebesgue's dominated convergence theorem.

We prove that \(u^{\epsilon}\rightharpoonup u\) in \(W^{1,\,p}(V)\) and (\ref{convergence of minimizer}).
For each \(0<\epsilon\le 1\), we note \(u^{\epsilon}-u\in W_{0}^{1,\,p}(V)\). By the Poincar\'{e} inequlatiy, we get
\[\lVert u^{\epsilon} \rVert_{L^{p}(V)}\le \lVert u\rVert_{L^{p}(V)}+C(n,\,p,\,V)\lVert \nabla u^{\epsilon}-\nabla u\rVert_{L^{p}(V)}\le C(n,\,p,\,V)\left(\lVert u\rVert_{W^{1,\,p}(V)}+\lVert\nabla u^{\epsilon}\rVert_{L^{p}(V)}\right).\]
By (\ref{convergence condition of p-th growth term}) and \(F_{V}^{\epsilon}\left(u^{\epsilon}\right)\le F_{V}^{\epsilon}(u)\) from (\ref{minimizer of approximated Fv}), we get
\begin{align*}
\lVert\nabla u^{\epsilon}\rVert_{L^{p}(V)}^{p}&\le \frac{1}{c_{1}}\int_{V}E_{p}\left(\nabla u^{\epsilon}\right)\,dx\le \frac{1}{c_{1}}\left(\beta\int_{V}\sqrt{\epsilon^{2}+\lvert\nabla u^{\epsilon}\rvert^{2}}\,dx+\int_{V}E_{p}^{\epsilon}\left(\nabla u^{\epsilon}\right)\,dx\right)\\&\le \frac{1}{c_{1}}\left( \beta\int_{V}\sqrt{\epsilon^{2}+\lvert\nabla u\rvert^{2}}\,dx+\int_{V}E_{p}^{\epsilon}(\nabla u)\,dx+\int_{V}f(u^{\epsilon}-u)\,dx\right)\\&\le C(n,\,p,\,q,\,\beta,\,c_{1},\,c_{2},\,V)\left(1+\int_{V}(1+\lvert\nabla u\rvert^{2})^{p/2}\,dx+\lVert f\rVert_{L^{q}(V)}\lVert \nabla (u^{\epsilon}-u)\rVert_{L^{p}(V)}\right)\\&\le \frac{\lVert\nabla u^{\epsilon}\rVert_{L^{p}(V)}^{p}}{2}+C(n,\,p,\,q,\,\beta,\,c_{1},\,c_{2},\,V)\left(1+\left(\lVert f\rVert_{L^{q}(V)}^{p^{\prime}}+1\right) \lVert\nabla u\rVert_{L^{p}(V)}^{p} \right).
\end{align*}
Here we have used the Sobolev embedding \(W_{0}^{1,\,p}(V)\hookrightarrow L^{q^{\prime}}(V)\) and the Young inequality.
Hence \(\left\{u^{\epsilon}\right\}_{0<\epsilon\le 1}\subset u+W_{0}^{1,\,p}(V)\) is bounded. Assume that \(u^{\epsilon_{N}}\rightharpoonup v\in u+W_{0}^{1,\,p}(V)\) for some sequence \(\{\epsilon_{N}\}_{N=1}^{\infty}\subset (0,\,1)\) such that \(\epsilon_{N}\to 0\) as \(N\to\infty\). We note that
\begin{equation}\label{Convergence of minmizer}
F_{V}(u)\le F_{V}\left(u^{\epsilon}\right)\le F_{V}^{\epsilon}\left(u^{\epsilon}\right)\le F_{V}^{\epsilon}(u)\to F_{V}(u)\textrm{ as } \epsilon\to 0
\end{equation}
by (\ref{minimizer of Fv})-(\ref{minimizer of approximated Fv}) and (\ref{condition for approximated energy functional}). By Lemma \ref{lowersemicontinuity of F}, we have
\[F_{V}(u)\le F_{V}(v)\le \liminf_{N\to\infty}F_{V} \left(u^{\epsilon_{N}}\right)=F_{V}(u),\]
which implies \(v=u\). Hence we obtain \(u_{\epsilon}\rightharpoonup u\,(\epsilon\to 0)\) in \(W^{1,\,p}(V)\). Again by (\ref{Convergence of minmizer}), we conclude (\ref{convergence of minimizer}).

By \cite[Proposition 3.32]{MR2759829} and the compact embedding \(W^{1,\,p}(V)\hookrightarrow L^{p}(V)\), we are reduced to showing that 
\[\limsup\limits_{\epsilon\to 0}\left\lVert\nabla u^{\epsilon}\right\rVert_{L^{p}(V)}\le\lVert\nabla u\rVert_{L^{p}(V)}\quad\textrm{up to a subsequence}\]
to complete the proof. By (\ref{elliptic condition for relaxed p-th growth term}), we can check that a smooth functional \({\hat E^{\epsilon}}(z)\coloneqq E^{\epsilon}(z)-C_{\ast}\left(\epsilon^{2}+\lvert z\rvert^{2}\right)^{p/2}\) is convex in \({\mathbb R}^{n}\) for sufficiently small \(C_{\ast}=C_{\ast}(c_{1},\,p)>0\). By (\ref{convergence condition of p-th growth term}), \({\hat E}(z)\coloneqq \lim\limits_{\epsilon\to 0}{\hat E^{\epsilon}}(z)=\beta\lvert z\rvert+E_{p}(z)-C_{\ast}\lvert z\rvert^{p}\) is also convex in \({\mathbb R}^{n}\). We note that 
\[C_{\ast}\left(\int_{V}\left(\epsilon^{2}+\left\lvert\nabla u\right\rvert^{2}\right)^{p/2}\,dx-\int_{V}\left(\epsilon^{2}+\left\lvert\nabla u^{\epsilon}\right\rvert^{2}\right)^{p/2}\,dx\right)=\int_{V}\left({\hat E^{\epsilon}}\left(\nabla u^{\epsilon}\right)-{\hat E^{\epsilon}}(\nabla u)\right)\,dx+\int_{V}f\left(u-u^{\epsilon}\right)\,dx+\left[F_{V}^{\epsilon}(u)-F_{V}^{\epsilon}\left(u^{\epsilon}\right)\right]\]
by the definitions of \({\hat E^{\epsilon}}\) and \(F_{V}^{\epsilon}\). 
For \(z_{0}\in{\mathbb R}\), we define 
\[h(z_{0})\coloneqq \nabla_{z}E_{p}(z_{0})+\left(\beta-pC_{\ast}\lvert z_{0}\rvert^{p-1}\right)\sgn(z_{0})\in{\mathbb R}^{n}, \quad \textrm{ where } \sgn(z_{0})\coloneqq \left\{\begin{array}{cc}
\displaystyle\frac{z_{0}}{\lvert z_{0}\rvert} & (z_{0}\not=0),\\ 
0 & (z_{0}=0).
\end{array} \right.\]
It is easy to check that \(h(z_{0})\in\partial {\hat E}(z_{0})\) for all \(z_{0}\in{\mathbb R}^{n}\). Moreover, we can check that for \(z_{0}\in{\mathbb R}^{n}\) and \(0<\epsilon\le 1\),
\begin{equation}\label{convergence of subgradient vector}
\nabla_{z}{\hat E^{\epsilon}}(z_{0})=\nabla_{z}E_{p}^{\epsilon}(z_{0})+\beta\displaystyle\frac{\lvert z_{0}\rvert\sgn(z_{0})}{\sqrt{\epsilon^{2}+\lvert z_{0}\rvert^{2}}}-pC_{\ast}\left(\epsilon^{2}+\lvert z_{0}\rvert^{2}\right)^{p/2-1}\lvert z_{0}\rvert \sgn(z_{0})\to h(z_{0}) \,(\epsilon\to 0).
\end{equation}
Since \(\partial {\hat E^{\epsilon}}(z_{0})=\left\{\nabla_{z}{\hat E^{\epsilon}}(z_{0}) \right\}\) for all \(z_{0}\in{\mathbb R}\), we have
\[{\hat E^{\epsilon}}\left(\nabla u^{\epsilon}\right)-{\hat E^{\epsilon}}(\nabla u)\ge \left\langle\nabla_{z}{\hat E^{\epsilon}}(\nabla u)\mathrel{}\middle|\mathrel{}\nabla u^{\epsilon}-\nabla u\right\rangle=\left\langle \nabla_{z}{\hat E^{\epsilon}}(\nabla u)-h(\nabla u)\mathrel{}\middle|\mathrel{}\nabla \left(u^{\epsilon}-u \right) \right\rangle+\left\langle h(\nabla u)\mathrel{}\middle|\mathrel{}\nabla \left(u^{\epsilon}-u \right) \right\rangle\quad \textrm{ a.e. in }V.\]
Hence we obtain
\begin{align}\label{perturbation estimate}
C_{\ast}\int_{V}\left(\epsilon^{2}+\lvert\nabla u\rvert^{2} \right)^{p/2}\,dx&\ge C_{\ast}\int_{V}\left\lvert\nabla u^{\epsilon}\right\rvert^{p}\,dx+\underbrace{\int_{V}f\left(u-u^{\epsilon}\right)\,dx}_{\coloneqq I_{1}(\epsilon)}+\underbrace{\left[F_{V}^{\epsilon}(u)-F_{V}^{\epsilon}\left(u^{\epsilon}\right)\right]}_{\eqqcolon I_{2}(\epsilon)}\nonumber\\&\quad +\underbrace{\int_{V}\left\langle \nabla_{z}{\hat E^{\epsilon}}(\nabla u)-h(\nabla u)\mathrel{}\middle|\mathrel{}\nabla \left(u^{\epsilon}-u \right) \right\rangle\,dx}_{\eqqcolon I_{3}(\epsilon)}+\underbrace{\int_{V}\left\langle h(\nabla u)\mathrel{}\middle|\mathrel{}\nabla \left(u^{\epsilon}-u \right) \right\rangle\,dx}_{\eqqcolon I_{4}(\epsilon)}.
\end{align}
We claim that, up to a subsequence,
\begin{equation}\label{claim for norm convergence}
\lim\limits_{\epsilon\to 0}I_{k}(\epsilon)=0\quad \textrm{ for all }k\in\{\,1,\,2,\,3,\,4\,\}.
\end{equation}
\(I_{1}(\epsilon)\to 0\,(\epsilon\to 0)\) up to a subsequence is clear by the compact embedding \(W^{1,\,p}(V)\hookrightarrow L^{q^{\prime}}(V)\). \(I_{2}(\epsilon)\to 0\,(\epsilon\to 0)\) follows from (\ref{convergence of minimizer}). \(h(\nabla u)\in L^{p^{\prime}}\left(V,\,{\mathbb R}^{n}\right)\) is clear by (\ref{growth estimate of nabla Ep}). Since \(\nabla u^{\epsilon}\rightharpoonup \nabla u\,(\epsilon\to 0)\) in \(L^{p}\left(V,\,{\mathbb R}^{n}\right)\), we obtain \(I_{4}(\epsilon)\to 0\,(\epsilon\to 0)\). From (\ref{convergence of subgradient vector}), it is clear that
\(\nabla_{z}{\hat E^{\epsilon}}(\nabla u)\to h(\nabla u) \,(\epsilon\to 0)\) a.e. in \(V\). By (\ref{uniformly bound estimate approximated Ep and nabla Ep at 0}) and (\ref{growth estimate of nabla Ep})-(\ref{growth estimate of nabla aproximated Ep}), we get
\begin{align*}
\left\lvert \nabla_{z}{\hat E^{\epsilon}}(\nabla u)-h(\nabla u)\right\rvert&\le C(p,\,C_{\ast})\left(\left\lvert\nabla_{z} E_{p}^{\epsilon}(\nabla u)-\nabla_{z} E_{p}^{\epsilon}(0)\right\rvert^{p^{\prime}}+\left\lvert\nabla_{z} E_{p}^{\epsilon}(0)\right\rvert^{p^{\prime}}+ \left\lvert\nabla_{z} E_{p}(\nabla u)\right\rvert^{p^{\prime}}+\left(1+\lvert\nabla u\rvert^{2}\right)^{p^{\prime}(p-1)/2}+\beta^{p^{\prime}}\right)\\&\le C(p,\,\beta,\,c_{2},\,C_{\ast})\left(\lvert\nabla u\rvert^{p}+1\right)\in L^{1}(V)\quad\textrm{ a.e. in }V,
\end{align*}
uniformly for \(0<\epsilon\le\epsilon_{0}.\)
From these, we conclude that \(\nabla_{z}{\hat E^{\epsilon}}(\nabla u)\rightarrow h(\nabla u)\) in \(L^{p^{\prime}}\left(V,\,{\mathbb R}^{n}\right)\) by Lebesgue's dominated convergence theorem.
We note that
\[\sup_{0<\epsilon\le 1}\left\lVert\nabla u^{\epsilon}-\nabla u \right\rVert_{L^{p}(V)}<\infty,\]
since we have already checked that \(\left\{u^{\epsilon}\right\}_{0<\epsilon\le 1}\subset u+W_{0}^{1,\,p}(V)\) is bounded in \(W^{1,\,p}(V)\).
Hence by the H\"{o}lder inequality, we deduce that \(I_{3}(\epsilon)\to 0\,(\epsilon\to 0)\). 
From (\ref{perturbation estimate})-(\ref{claim for norm convergence}), by letting \(\epsilon\to 0\) we obtain
\[C_{\ast}\limsup_{\epsilon\to 0}\left\lVert \nabla u^{\epsilon}\right\rVert_{L^{p}(V)}^{p}=C_{\ast}\limsup_{\epsilon\to 0}\int_{V}\left\lvert\nabla u^{\epsilon}\right\rvert^{p}\,dx\le C_{\ast}\limsup_{\epsilon\to 0} \int_{V}\left(\epsilon^{2}+\lvert\nabla u\rvert^{2} \right)^{p/2}\,dx=C_{\ast}\lVert \nabla u\rVert_{L^{p}(V)}^{p}\]
up to a subsequence.
Here we have used Lebesgue's dominated convergence theorem for the last equality.
This completes the proof. 
\end{proof}
\subsection{A Fatou-type estimate}
Lemma \ref{weak+uniform bound} is used in the proof of main theorem in Section \ref{Subsect Local approximation}.
\begin{lemma}\label{weak+uniform bound}
Let \(\left(E, \lVert\,\cdot\,\rVert_{E}\right)\) be a Banach space and \(X\subset{\mathbb R}^{n}\) be a \({\mathcal L}^{n}\)-measurable set. Suppose that sequences \(\{u_{N}\}_{N=1}^{\infty}\subset L^{p}(X, E)\,(1\le p<\infty)\), \(\{C_{N}\}_{N=1}^{\infty}\subset \lbrack0,\,\infty)\) satisfy 
\begin{equation}\label{uniform ess bound}
\esssup\limits_{x\in X}\,\lVert u_{N}(x)\rVert_{E}\le C_{N}\quad \text{ for all } N,
\end{equation}
\begin{equation}\label{weak convergent}
u_{N}\rightharpoonup u \quad \text{\rm in } L^{p}(X,\,E) \quad \text{ as }N\to\infty
\end{equation}
for some \(u\in L^{p}(X,\, E)\). Then we have
\begin{equation}\label{ess bound}
\esssup\limits_{x\in X}\,\lVert u(x)\rVert_{E}\le \liminf\limits_{N\to\infty} C_{N}.
\end{equation}
\end{lemma}
\begin{proof}
We may assume that \(C_{\infty}\coloneqq\liminf\limits_{N\to\infty}C_{N}<\infty\), since otherwise (\ref{ess bound}) is clear. 
Since \({\mathcal L}^{n}\) is \(\sigma\)-finite, it suffices to show that \({\mathcal L}^{n}(X_{\epsilon,\, r})=0\) for all \(\epsilon, r >0\), where \[X_{\epsilon,\, r}\coloneqq\{x\in X\mid \lvert x\rvert\le r\textrm{ and } \lVert u(x)\rVert_{E}>C_{\infty}+\epsilon\}\subset X\] is a \({\mathcal L}^{n}\)-measurable set. By (\ref{weak convergent}), it is clear that
\[u_{N}\rightharpoonup u \quad \textrm{ in } L^{p}(Y,\,E) \, (N\to\infty)\]
for any fixed \({\mathcal L}^{n}\)-measurable set \(Y\subset X\). Hence for each fixed \(\epsilon,\,r>0\), we obtain 
\begin{align*}
(C_{\infty}+\epsilon)\left({\mathcal L}^{n}(X_{\epsilon,\,r})\right)^{1/p}
&\le \left(\int_{X_{\epsilon,\,r}}\lVert u(x)\rVert_{E}^{p}\,dx\right)^{1/p} \quad \left(\text{by the definiton of } X_{\epsilon,\,r}\right)\\
&\le\liminf\limits_{N\to\infty}\left(\int_{X_{\epsilon,\,r}}\lVert u_{N}(x)\rVert_{E}^{p}\,dx\right)^{1/p} (\textrm{since the norm map is weakly lower semi-continuous})\\
&\le \left({\mathcal L}^{n}(X_{\epsilon,\,r})\right)^{1/p}\liminf\limits_{N\to\infty}C_{N}=C_{\infty}\left({\mathcal L}^{n}(X_{\epsilon,\,r})\right)^{1/p}\quad \left(\textrm{by (\ref{uniform ess bound})}\right).
\end{align*}
Since \({\mathcal L}^{n}(X_{\epsilon,\,r})\le {\mathcal L}^{n}\left(\left\{x\in {\mathbb R}^{n}\mathrel{}\middle|\mathrel{} \lvert x\rvert\le r\right\}\right)<\infty\), this implies \({\mathcal L}^{n}(X_{\epsilon,\,r})=0\), which completes the proof of (\ref{ess bound}).
\end{proof}

\end{document}